\documentclass[12pt,notitlepage]{report}
\title{One sided extendability and $p$-continuous analytic capacities}
\author{E. Bolkas, V. Nestoridis, C. Panagiotis and M. Papadimitrakis}
\date{}
\usepackage[english]{babel}
\usepackage[utf8]{inputenc}
\usepackage{titlesec}
\usepackage{amsthm}
\usepackage{amsmath}
\usepackage{amsfonts}
\usepackage{mathrsfs}
\usepackage{enumerate}

\titleformat*{\section}{\normalsize\bfseries}
\titleformat*{\subsection}{\Large\bfseries}

\titleformat*{\section}{\normalsize\bfseries}
\titleformat*{\subsection}{\Large\bfseries}
\newtheorem{theorem}{Theorem}[chapter]
\numberwithin{theorem}{section}
\newtheorem{defi}[theorem]{Definition}
\newtheorem{lemma}[theorem]{Lemma}
 
\newtheorem{prop}[theorem]{Proposition}

\newtheorem{corol}[theorem]{Corollary}

\theoremstyle{definition}

\newtheorem{remark}[theorem]{Remark}

\DeclareMathOperator\supp{supp}
\addto\captionsenglish{}

\begin{document}
\maketitle

\noindent 
Dedicated to the memory of Professor Jean-Pierre Kahane.

\begin{abstract}
\noindent 
Using complex methods combined with Baire's Theorem we show that one-sided extendability, 
extendability and real analyticity are rare phenomena on various 
spaces of functions in the topological sense. These considerations led us to introduce the 
$p$-continuous analytic capacity and variants of it, $p\in \{0,1,2, \cdots\}\cup 
\{\infty\}$, for compact or closed sets in $\mathbb{C}$. We use these capacities in order to 
characterize the removability of singularities of functions in the spaces $ A^p $.
\end{abstract}

\noindent AMS Classification numbers: $30H05$, $53A04$, $30C85$\\


\section{Introduction}
In \cite{[1]} it is proven that the set $ X $ of nowhere analytic functions in $ C^{\infty}
([0,1]) $ contains a dense and $ G_{\delta} $ subset of $ 
C^{\infty}([0,1]) $. In \cite{[2]} using Fourier methods it is shown that $ X $ is itself a 
dense and $ G_{\delta} $ subset of $ C^{\infty}([0,1]) $. 
Furthermore, combining the above methods with Borel's Theorem (\cite{[3]}) and a version of 
Michael's Selection Theorem (\cite{[4]}) the above result has 
been extended to $ C^{\infty}(\gamma) $, where $ \gamma $ is any analytic curve. In the case where $\gamma$ is the unit circle $ T $ every function $ f \in 
C^{\infty}(T) $ can be written as a sum $ f=g+w $ where $ g $ belongs to $ A^{\infty}(D) $ and is holomorphic on the open unit disc $ D $ and very smooth up 
to the boundary and $ w $ has similar properties in $D^c$. Now if we assume that $f$ is extendable somewhere towards one side of $T$, say in $D^c$, then 
because $w$ is regular there, it follows that $g\in A^{\infty}(D)$ is extendable. But the phenomenon 
of somewhere extendability has been proven to be a rare phenomenon in the Fr{\'e}chet space $A^{\infty}(D)$ (\cite{[5]}). It follows that the phenomenon of one 
sided somewhere extendability is a rare phenomenon in $C^{\infty}(T)$ or more generally in $C^p(\gamma)$, $p\in \{\infty\}\cup \{0,1,2, \cdots\}$ for any 
analytic curve $ \gamma $ (\cite{[2]}).\\
After the preprint \cite{[2]} has been circulated, P. Gautier noticed that the previous result holds more generally for Jordan arcs without the 
assumption of analyticity of the curve. Indeed, applying complex methods appearing in the last section of \cite{[2]} we prove this result. It suffices to 
use the Oswood-Caratheodory Theorem combined with Montel's Theorem and the Poisson integral formula applied to the boundary values of bounded holomorphic 
functions in $H^{\infty}(D)$. In fact this complex method is most natural to our considerations of extendability, real-analyticity and one 
sided-extendability. The proofs are simplified and the results hold under much more general assumptions than the assumptions imposed by the Fourier method. 
This complex method is developed in the present paper.\\
In section $4$ we prove that extendability and real analyticity are rare phenomena in various spaces of functions on locally injective curves $\gamma$. For 
the real analyticity result we assume that $\gamma$ is analytic and the result holds in any $C^{k}(\gamma), k\in \{0,1,2,3, \cdots\}\cup \{\infty\}$ 
endowed with its nature topology. For the other results the phenomena are proven to be rare in $C^k(\gamma)$ provided that the locally injective curve 
$\gamma$ has smoothness at least of degree $k$.\\
In section $5$ initially we consider a finite set of disjoint curves $\gamma_1, \gamma_2, \cdots , \gamma_n$. Then in the case where $\gamma_1, \gamma_2, \cdots , \gamma_n$ are disjoint Jordan curves in $\mathbb{C}$ bounding a domain $\Omega$ of finite connectivity 
we consider the spaces $A^p(\Omega), p\in \{0,1,2,3, \cdots\}\cup \{\infty\}$ which by the maximum principle can be seen as function spaces on $ 
\partial\Omega=\gamma_1^*\cup\cdots\cup\gamma_n^* $. In these spaces we show that the above phenomena of extendability 
or real analyticity are rare phenomena. For the real analyticity result we assume analyticity of $ \partial\Omega $, but for the extendability result we do 
not need to assume any smoothness of the boundary.\\
In section $6$ we consider the one sided extendability from a locally injective curve $\gamma$ and we prove that this is a rare phenomenon in various spaces 
of functions. We construct a denumerable family $ G_n $ of Jordan domains $ G $ containing in their boundary a non-trivial arc $J$ of the image $\gamma^*$ 
of $\gamma$, such that each other domain $\Omega$ with similar properties contains some $G_n$. We show that the phenomenon of extendability is rare for 
each domain $G$. Then by denumerable union (or intersection of the complements) we obtain our result with the aid of Baire's Category Theorem. We mention 
that the one-sided extendability of a function $f:\gamma^*\rightarrow \mathbb{C}$ is meant as the existence of a function $F:G \cup J \rightarrow 
\mathbb{C}$ which is hololomorphic on the Jordan domain $G$, continuous on $G \cup J$ and such that on the arc $ J $ of $ \gamma^* $ we have $ F|_J=f|_J $. 
Such notions of one-sided extendability have been considered in \cite{[6]} and the references therein, 
but in the present article and \cite{[2]} it is, as far as we know, the first time where the phenomenon is proven to be rare.\\
At the end of section $6$ we prove similar results on one-sided extendability on the space $A^p(\Omega)$, where $\Omega$ is a finitely connected domain in 
$\mathbb{C}$ bounded by a finite set of disjoint Jordan curves $\gamma_1, \gamma_2, \cdots , \gamma_n$. Now the extension $F$ of a function $f\in 
A^p(\Omega)$ has to coincide with $f$ only on a non-trivial arc of the boundary of $\Omega$, not on an open subset of $\Omega$. Certainly if the continuous 
analytic capacity of $\partial\Omega$ is zero, the latter automatically happens, but not in general.\\
In section $7$ we consider a domain $\Omega$  in $ \mathbb{C} $, a compact subset $L$ of $\Omega$ and we study the phenomenon of extendability of a 
function $f\in A^p(\Omega\setminus L)$ to a function $ F $ in $ A^p(\Omega) $. There is a dichotomy. Either for every $f$ this is possible or generically 
for all $f \in A^p(\Omega\setminus L)$ this fails. In order to characterize when each horn of the above dichotomy holds we are led to define the 
$p$-continuous analytic capacity $ a_p(L) $ ($p\in \{0,1,2,3, \cdots\}\cup \{\infty\}$), where $a_0(L)$ is the known continuous analytic capacity $ 
a_0(L)=a(L) $ (\cite{[7]}).\\
The study of the above capacities and variants of it is done in section $3$. For $p=1$ the $p$-continuous analytic capacity $ a_1 $ is distinct from the continuous analytic capacity $a_0=a$. In particular
if $K_{1/3}$ is the usual Cantor set lying on $[0,1]$ and $L=K_{1/3}\times K_{1/3}$, then $a_0(L)>0 $, but $a_1(L)=0$. This means that for any open set $U$ containing $L$ there exists a function in $A(U\setminus L )$ which is not holomorhic on $U$, but if the derivative of a function in $A(U \setminus L)$ extends continuously on $L$, then the function is holomorphic on $U$. Generic versions of this fact imply that $ A^1(U\setminus L) $ is of first category in $ A^0(U\setminus L)=A(U\setminus L) $.  If we replace the spaces $A^p$ with the ${\tilde{A}}^p$ spaces of Whitney type, then the extension on $L$ is equivalent to the fact that the interior of $L$ is void. Thus we can define the continuous analytic capacities ${\tilde{a}}^p(L)$ which  vanish if and only if the interior of $L$ is empty. We prove what is needed in section $7$. More detailed study of those capacities will be done in future papers; for instance we can investigate the semiadditivity of $a_p$, whether the vanishing of $a_p$ on a compact set $L$ is a local phenomenon and whether replacing the continuous analytic capacity $a$ by the Ahlfors analytic capacity $\gamma$ we can define capacities $\gamma_p$ satisfying the analogous properties. Certainly the spaces $ A^p(\Omega) $ will be replaced by $ H^{\infty}_p(\Omega) $, the space of holomorphic functions on $\Omega$ such that for every $l\in \mathbb{N}$, $l\leq p$ the derivative $f^{(l)}$ of order $l$ is bounded on $\Omega$. We will also examine if a dichotomy result as in section $7$ holds for the spaces $H^{\infty}_p(\Omega)$ in the place of $A^p(\Omega)$. All these in future papers.\\
In section $2$ some preliminary geometry of locally injective curves is presented; for instance a curve is real analytic if and only if real analyticity of a function on the curve is equivalent to holomorphic extendability of the function on discs centred on points of the curve.
We also show that if the map $\gamma$ defining the curve is a homeomorphism with non-vanishing derivative, then the spaces $C^k(\gamma), k\in \{0,1,2,3, 
\cdots\}\cup \{\infty\}$ are independent of the particular parametrization $ \gamma $ and depend only on the image $ \gamma^* $ of $\gamma$. Thus in some 
cases it makes sense to write $C^k(\partial\Omega)$ and prove generic results in these spaces.\\
Finally, we mention that some of the results of section $5$ are valid for analytic curves $ \gamma $; that is, they hold when we use a conformal 
parametrization of $ \gamma $. Naturally comes the question whether these results remain true if we change the parametrization of the curve; in particular 
what happens if we consider the parametrization with respect to the arc length $s$? Answering this question was the motivation of \cite{[9]}, \cite{[10]} where 
it is proven that arc length is a global conformal parameter for any analytic curve. Thus the results of section $5$ remain true if we use the arc length 
parametrization.
Finally, we mention that in the present paper we start with qualitative categorical results, which lead us to quantitative notions as the $p$-continuous analytic capacity $a_p$ and the $p$-analytic capacity $\gamma_p$.

\section{Preliminaries}

In most of our results it is important what is the degree of smoothness of a curve and the relation of real analyticity of functions on a curve with the 
holomorphic extendability of them around the curve. That is why we present here some basic results concerning locally injective curves in $ \mathbb{C} $. 

Unless otherwise specified $I$ is an interval and $X$ is an interval or the unit circle.

\begin{defi}
Let $ \gamma:X\rightarrow \mathbb{C} $ be a continuous and locally injective function and $ l\in 
\{1,2,3, \cdots\}\cup \{\infty\} $. The curve $ \gamma $ belongs to the class $ C^{l}(X) $,  if for $ k\in \{1,2,3, \cdots\}, k\leq l $, the derivative $ 
\gamma^{(k)} $ exists and is a continuous function. 
\end{defi}

\begin{defi}
Let $ \gamma:X\rightarrow \mathbb{C} $ be a continuous and locally injective function and $ k\in 
\{1,2,3, \cdots\}\cup \{\infty\} $. A function $ f:\gamma^*\rightarrow \mathbb{C} $ defined on the image $ \gamma^*=\gamma(X) $ belongs to the class $ 
C^{k}(\gamma) $ if for every $l\in \{1,2,3,\cdots,\}$, $l\leq k$ the derivative $ (f\circ \gamma)^{(l)} $ exists and is a continuous function. Let $ (X_n), n\in 
\{1,2,3, \cdots\} $ be an increasing sequence of compact intervals such that $ \bigcup\limits_{n=1}^{\infty}X_n=X $. The topology of the space $ 
C^{k}(\gamma) $ is defined by the seminorms 
$$\sup\limits_{t\in X_n}{|(f\circ\gamma)^{(l)}(t)|},l=0,1,2,\dots,k, n=1,2,3, \dots $$ 
In this way $C^k(\gamma)$ becomes a Banach space if $k<+\infty$ and $X$ is compact. Otherwise is a Fr{\'e}chet space. In every case Baire's theorem is at our disposal.
\end{defi}

\begin{defi}
Let $ \gamma:I\rightarrow \mathbb{C} $ be a continuous and locally injective function. We will say that the curve $ \gamma$ is 
analytic at $ t_0 \in I$ if there exist an open set $ t_0\in V\subseteq \mathbb{C} $, a real number $ \delta>0 $ with $ (t_0-\delta,t_0+\delta)\cap I\subset 
V $ and a holomorphic and injective function $ F:V\rightarrow \mathbb{C} $ such that $ F|_{(t_0-\delta,t_0+\delta)}=\gamma|_{(t_0-\delta,t_0+\delta)} $. If 
$ \gamma$ is analytic at every $ t\in I $ we will say that $ \gamma$ is an analytic curve. 
\end{defi}

\begin{lemma}
Let $ t_0 \in I $ and $ \gamma:I\rightarrow \mathbb{C} $ be a continuous and locally injective function.\\
We suppose that for every function $ f: I \rightarrow \mathbb {C} $ items \eqref{real} and \eqref{complex} are equivalent:
\begin{enumerate}
\item \label{real} There exists a power series of real variable
 $$\sum\limits_{n=0}^\infty a_n(t-t_0)^n,a_n \in \mathbb{C}$$  
 with a positive radius of convergence $r>0$ and there exists $0<\delta\leq r$ such that $$f(t)=
 \sum\limits_{n=0}^\infty a_n(t-t_0)^n$$ for $t\in(t_0-\delta,t_0+
 \delta)\cap I$.\\
\item \label{complex} There exists a power series of complex variable 
 $$\sum\limits_{n=0}^\infty b_n(z- \gamma(t_0))^n,b_n \in \mathbb{C}$$ with a positive radius of convergence $s>0$ and $0<\epsilon\leq s$ such that 
 $$f(t)=
 \sum\limits_{n=0}^\infty b_n(\gamma(t)-\gamma(t_0))^n$$ for $t\in(t_0-   
 \epsilon,t_0+
 \epsilon)\cap I$.
\end{enumerate}

Then $ \gamma $ is analytic at $ t_0 $.
\end{lemma}

\begin{proof}
Implication \eqref{complex} $\Rightarrow $ \eqref{real} will only be used to prove that $ \gamma $ is differentiable on an open interval that contains $ t_0 $. We start by considering $ 
\beta >0 $ and $J=(t_0-\beta, t_0+\beta)\cap I $. For every $ t \in J $ we choose $ f(t)=\gamma(t)= \gamma(t_0)+(\gamma(t)-\gamma(t_0)) $ and so by \eqref{complex} $\Rightarrow $ \eqref{real} we obtain that there exists $0< \delta < \beta$ such that $$\gamma(t)=
 \sum\limits_{n=0}^\infty a_n(t-t_0)^n$$ for some constants $a_n \in \mathbb{C}$ and for every $t\in(t_0-\delta,t_0+
 \delta)\cap I$. Therefore $ \gamma $ is differentiable on this interval. Now for $ g(t)=t=t_0+(t-t_0) $ by \eqref{real} $\Rightarrow $ \eqref{complex} we have that there is  
$0<\epsilon\leq \delta$ such that $$t= \sum\limits_{n=0}^\infty b_n(\gamma(t)-\gamma(t_0))^n$$ for some constants $b_n \in \mathbb{C}$ for every $t\in(t_0- \epsilon,t_0+ \epsilon)$. Differentiation of the above equation at $t=t_0 $ yields the relation $ 1=b_{1}\gamma'(t_0) $ which implies that $b_{1}\neq 0$. The power series  $\sum\limits_{n=0}^\infty 
b_n(z-\gamma(t_0))^n$ has a positive radius of convergence and so there exists $\alpha>0$ such that $ \gamma(t)\in D(\gamma(t_0),\alpha)$ for every $ t \in 
(t_0- \epsilon,t_0+ \epsilon)\cap I $ and the function $ f:D(\gamma(t_0),\alpha)\rightarrow \mathbb{C}$ with $$ f(z)=\sum\limits_{n=0}^\infty 
b_n(z-\gamma(t_0))^n $$ is holomorphic. Also, $ f(\gamma(t))=t $ for every $ t \in (t_0- \epsilon,t_0+ \epsilon)\cap I $. Since $ 
f'(\gamma(t_0))=b_1\neq 0 $ , $ f $ is locally invertible and thus the inverse, $h$, of $f$ is well defined on the open disk $ D(t_0,\eta) $ where $ 0<\eta <\epsilon $. Moreover, $ 
\gamma(t)=h(t) $ for every $ t \in (t_0-\eta, t_0+\eta)\cap I $ and $h$ is holomorphic and injective. 
\end{proof}

\begin{remark}
The above proof shows that if $\gamma $ in Lemma $2.4$ belongs to $ C^1(I) $, then the conclusion of the lemma is true even if we only assume that \eqref{real} $\Rightarrow $ \eqref{complex} is true. 
\end{remark}
 
The following lemma is the inverse of Lemma $2.4$.
 
\begin{lemma}
Let $ t_0 \in I $ and $ \gamma:I\rightarrow \mathbb{C} $ be a continuous and locally injective function, which is analytic at $ t_0 $, and $ f:I\rightarrow \mathbb{C} $. Then the following are equivalent: 
\begin{enumerate}
\item \label{real variable} There exists a power series of real variable
 $$\sum\limits_{n=0}^\infty a_n(t-t_0)^n,a_n \in \mathbb{C}$$ with positive radius of convergence $r>0$ and there exists a $0<\delta\leq r$ such that 
$$f(t)=
 \sum\limits_{n=0}^\infty a_n(t-t_0)^n$$ for $t\in(t_0-\delta,t_0+
 \delta)\cap I$.
 
\item \label{complex variable} There exists a power series with complex variable
 $$\sum\limits_{n=0}^\infty b_n(z-\gamma(t_0))^n,b_n \in \mathbb{C}$$ with positive radius of convergence $s>0$ and there exists $\epsilon>0$ such that 
 $$f(t)=
 \sum\limits_{n=0}^\infty b_n(\gamma(t)-\gamma(t_0))^n$$ for $t\in(t_0-   
 \epsilon,t_0+
 \epsilon)\cap I$.
\end{enumerate}
\end{lemma}

\begin{proof} 
Let us start with the following observation. Since $\gamma$ is an analytic curve at $ t_0 $ there is an open disk $ D(t_0,\varepsilon)\subseteq\mathbb{C} $, where $ \varepsilon>0 $ and a 
holomorphic and injective function $ \Gamma:D(t_0,\varepsilon)\rightarrow \mathbb{C} $ with $\Gamma(t)=\gamma(t)$ for $t \in 
(t_0-\varepsilon,t_0+\varepsilon)\cap I$.\\

\eqref{real variable} $\Rightarrow $ \eqref{complex variable}
We define the function 
$$g(z)=\sum\limits_{n=0}^\infty a_n(z-t_0)^n,$$
$z\in D(t_0,\delta)$,
which is well defined and holomorphic on $D(t_0,\delta)$. As
$$\Gamma^{-1}: \Gamma(D(t_0,\varepsilon))\rightarrow \mathbb{C}$$ is a holomorphic function we are led to consider the functions $$F=g\circ \Gamma^{-1}:
\Gamma(D(t_0,\varepsilon))\rightarrow \mathbb{C},$$ (where $\Gamma(D(t_0,\varepsilon))$ is an open set) and 
$$(F\circ\Gamma)(t)=f(t),t\in D(t_0,\varepsilon)
\cap I.$$ It follows that there are $b_n\in \mathbb{C}$, $n\in\mathbb{N}$ and $ \delta>0 $ such that $$ F(z)=\sum\limits_{n=0}^\infty b_n(z-\gamma(t_0))^n $$
for every  $z \in D(\gamma(t_0),\delta)\subseteq\Gamma(D(t_0,\varepsilon))$. Thus $$ f(t)=(F
\circ\gamma)(t)=\sum\limits_{n=0}^\infty b_n(\gamma(t)-\gamma(t_0))^n $$ in an interval $ (t_0- s,t_0+s)\cap I$ where $ s>0 $.
\\

\eqref{complex variable} $\Rightarrow $ \eqref{real variable}
We consider the function $$G(z)=\sum\limits_{n=0}^\infty b_n(z-\gamma(t_0))^n,$$ 
$z\in D(\gamma(t_0),s)$. Choose $ a>0 $ with $ a<\varepsilon $ such that $ \Gamma(D(t_0,a))\subseteq D(\gamma(t_0),s)$.
The function $$G\circ\Gamma :D(t_0,a)\rightarrow \mathbb{C}$$ is holomorphic. 
Therefore there are
$a_n\in \mathbb{C}$, $n\in\mathbb{N}$ such that
$$(G\circ\Gamma)(z)=\sum\limits_{n=0}^\infty a_n(z-t_0)^n,
z\in D(t_0,a)$$ and consequently  
$$f(t)=G(\gamma(t))=\sum\limits_{n=0}^\infty a_n(t-t_0)^n,t\in(t_0-
a,t_0+a).$$
\end{proof}

\begin{defi}
Let $\gamma : I \rightarrow \mathbb{C}$ be a locally injective curve and 
$z_0=\gamma(t_0),t_0 \in I$. A function $f: \gamma^* \rightarrow \mathbb{C}$ belongs to the class of non-holomorphically extendable at $(t_0,z_0=\gamma(t_0))$ functions if  
there are no open disk 
$D(z_0,r), r>0$ and $\eta >0$ and a holomorphic function $F:D(z_0,r)\rightarrow 
\mathbb{C}$, such that $\gamma((t_0-\eta,t_0+\eta)\cap I) \subset
D(z_0,r)$ and $ F(\gamma(t))=f(\gamma(t)) $ for all $t\in(t_0-\eta,t_0+\eta)\cap I$. 
Otherwise we say that $f$ is holomorphically extendable at $(t_0,z_0=\gamma(t_0))$.
\end{defi}

\begin{defi}
Let $\gamma : I \rightarrow \mathbb{C}$ be a continuous map and $t_0 \in I$. A function $ f:\gamma^{*}\rightarrow \mathbb{C} $ is real 
analytic at $(t_0,z_0=\gamma(t_0))$ if there exist $ \delta>0 $ and a power series $ \sum\limits_{n=0}^\infty a_n(t-t_0)^n $ with a radius of convergence  $ \epsilon>\delta>0 $, such that $ f(\gamma(t))=\sum\limits_{n=0}^\infty a_n(t-t_0)^n $ for every $ t \in(t_0-\delta,t_0+\delta)\cap I $.
\end{defi}

The following proposition associates the phenomenon of real-analyticity and that of holomorphically extendability.

\begin{prop}
Let $\gamma : I \rightarrow \mathbb{C}$ be an analytic curve at $ t_0 $ and $t_0 \in I$. A function $ f:\gamma^{*}\rightarrow \mathbb{C} $ is real 
analytic at $(t_0,z_0=\gamma(t_0))$ if and 
only if $ f $ is holomorphically extendable at $(t_0,z_0=\gamma(t_0))$.
\end{prop}

\begin{proof}
At first we will prove direction $ \Rightarrow: $. If $ f $ is real analytic at $(t_0,z_0=\gamma(t_0))$, then from Lemma $2.6$
$$f(\gamma(t))=\sum\limits_{n=0}^\infty b_n(\gamma(t)-\gamma(t_0))^n$$ for every $t\in(t_0-\epsilon,t_0+ \epsilon)\cap I$ and for some $ b_n \in \mathbb {C}, \epsilon>0 $. From the continuity of $\gamma$, there exists $\eta>0$ such that $\gamma((t_0-\eta,t_0+\eta)\cap I)\subset D(z_0,\epsilon)$. Therefore the function $$F(z)=\sum\limits_{n=0}^\infty b_n(z-\gamma(t_0))^n$$ defined on $D(z_0,\epsilon)$ is equal to $f$ on $\gamma((t_0-\eta,t_0+\eta)\cap I)$. Thus the function $f$ is holomorphically extendable at $(t_0,z_0=\gamma(t_0))$.\\
Next we prove direction $ \Leftarrow: $ If $ f $ is extendable at $(t_0,z_0=\gamma(t_0))$, then there are $r>0$ and a holomorphic function $F :D(\gamma(t_0),r) \rightarrow \mathbb{C}$ such that $$f(\gamma(t))=F(\gamma(t))$$ for every $t\in(t_0-\epsilon,t_0+\epsilon)\cap I$ and for 
some $\epsilon>0 $. Let $$\sum\limits_{n=0}^\infty b_n(z-\gamma(t_0))^n$$ be the Taylor expansion
of the holomorphic function $F$. It follows that $$f(\gamma(t))=F(\gamma(t))=\sum\limits_{n=0}^\infty b_n(\gamma(t)-\gamma(t_0))^n$$ for every $t\in(t_0-\epsilon,t_0+\epsilon)\cap I$ 
and as a result, again from Lemma $2.6$, we conclude that $ f $ is real analytic at 
$(t_0,z_0=\gamma(t_0))$, because the curve $\gamma$ is analytic at $ t_0 $.
\end{proof}

The following theorem is a consequence of Lemma $2.4$ and Proposition $2.9$.
\begin{theorem}
Let $\gamma : I \rightarrow \mathbb{C}$ be a continuous and locally injective curve and $t_0 \in I$. Then $ \gamma $ is analytic at $ t_0 $ if and only if for every function $ f:\gamma^*\rightarrow \mathbb{C} $ the following are 
equivalent:
\begin{enumerate}
\item $ f $ is real analytic at $ (t_0,z_0=\gamma(t_0)) $.
\item $ f $ is holomorphically extendable at $ (t_0,z_0=\gamma(t_0)) $.
\end{enumerate}
\end{theorem}

Now we will examine a different kind of differentiability.

\begin{defi}
Let $\gamma : X \rightarrow \mathbb{C}$ be a continuous and injective curve. We define the derivative of a function $ f:\gamma^*\rightarrow \mathbb{C} $ at $ \gamma(t_0) $, where $ t_0 \in X $ by 
$$ \dfrac{{df}}{dz}(\gamma(t_0))=\lim_{t\rightarrow t_0}\dfrac{f(\gamma(t))-f(\gamma(t_0))}{\gamma(t)-\gamma(t_0)} $$
if the above limit exists and is a complex number. Inductively, we define $$\dfrac{{d^kf}}{dz^k}(\gamma(t))=\dfrac{{d\left(\dfrac{{d^{k-1}}f}{dz^{k-1}}(\gamma(t))\right)}}{dz}(\gamma(t)).$$
\end{defi}

\begin{remark}
If $\gamma : X \rightarrow \mathbb{C}$ is a homeomorphism onto $\gamma^*$, then we can equivalently define  $ \dfrac{{df}}{dz}(\gamma(t_0)) $ as
$$ \dfrac{{df}}{dz}(\gamma(t_0))=\lim_{z\in \gamma^*, z\rightarrow \gamma(t_0)}\dfrac{f(z)-f(\gamma(t_0))}{z-\gamma(t_0)} $$
if the above limit exists and is a complex number.
\end{remark}

\begin{defi}
Let $\gamma : X \rightarrow \mathbb{C}$ is a homeomorphism onto $\gamma^*$. We define the set $ C^k(\gamma^*) $ inductively. A function $ f:\gamma^*\rightarrow \mathbb{C} $ belongs to the class $ C^1(\gamma^*) $ if $\dfrac{{df}}{dz}(\gamma(t))$ exists and is continuous for $ t\in I $.
For $ k\in \{2,3, \cdots\} $ a function $ f:\gamma^*\rightarrow \mathbb{C} $ belongs to the class if 
$$\dfrac{{d^kf}}{dz^k}(\gamma(t))=\dfrac{{d\left(\dfrac{{d^{k-1}}f}{dz^{k-1}}(\gamma(t))\right)}}{dz}(\gamma(t))$$ exists and is continuous for $ t\in X $. 
Finally, a function $ f:\gamma^*\rightarrow \mathbb{C} $ belongs to the class $ C^{\infty}(\gamma^*) $ if $\dfrac{{d^kf}}{dz^k}(\gamma(t))$ exists and is 
continuous for every $ t\in X $ and for every $ k\in \{1,2,3, \cdots\} $.
Let also $ (X_n), n\in \{1,2,3, \cdots\} $ be an increasing sequence of compact subsets of $ \gamma^* $ such that $ 
\bigcup\limits_{n=1}^{\infty}X_n=\gamma^* $. For $k\in \{1,2,3, \cdots\}\cup \{\infty\}$ the topology of the space $ C^{k}(\gamma^*) $ is defined by the 
seminorms $$\sup\limits_{z\in X_n}{|\dfrac{{d^{l}f}}{dz^l}(z)|},l=0,1,2,\dots,k, n=1,2,3, \dots $$ 
In this way $C^k(\gamma^*)$ becomes a Banach space if $ k<\infty $ and $ I $ is compact. Otherwise $C^k(\gamma^*)$ is a Fr{\'e}chet space.
\end{defi}

\begin{prop}
Let $\gamma : X \rightarrow \mathbb{C}$ is a homeomorphism onto $\gamma^*$ with $ \gamma'(t)\neq 0 $ for every $ t \in X $ and 
$ k\in \{1,2,3, \cdots\}\cup \{\infty\} $. If $ C^{k}(X)=C^{k}(\gamma^*)\circ \gamma $, then $ \gamma \in C^k(X) $.
\end{prop}

\begin{proof}
The function $ f:\gamma^*\rightarrow \mathbb{C} $ with $ f(\gamma(t))=\gamma(t) $ for $ t\in X $ belongs to the class $ C^k(\gamma^*) $ and therefore the 
function $ \gamma=f\circ \gamma :X\rightarrow \mathbb{C} $ belongs to the class $ C^{k}(X)$.
\end{proof}

Now we will prove the inverse of the previous proposition: If $ \gamma \in C^k(X) $, then $ C^{k}(X)=C^{k}(\gamma^*)\circ \gamma $. In order to do that we need the following lemma which will also be useful to us later.

\begin{lemma}
Let $X$ be an interval $I\subset \mathbb{R}$ or the unit circle $T$, $\gamma\in C^k(X)$, 
$k\in\{1,2,\dots\}\cup\{\infty\}$ and 
$f\in C^k(\gamma^*)$ and $g=f\circ\gamma$. Then $ g\in C^k(X) $ and there exist polynomials
$P_{j,i}$, $i,j\in\{1,2,3,\dots\}$, $j\leq i \leq k$, defined on $\mathbb{C}^i$, such that
$$g^{(i)}(t)=\sum\limits_{j=1}^i \dfrac{{d^j f}}{dz^j}(\gamma(t))P_{j,i}(\gamma^{'}(t),
\gamma^{''}(t),\dots,\gamma^{(i)}(t))$$
where the derivatives of $ \gamma $ are with respect to the real variable $ t $ in the case $ X=I $ and with respect to the complex variable $ t $, $ |t|=1 $ in the case $ X=T $.
\end{lemma}

\begin{proof}
We will prove the lemma by induction on $i$. For $i=1\leq k$, 
$$g^{'}(t)=(f\circ\gamma)^{'}(t)=\lim_{t\rightarrow t_0}\dfrac{f(\gamma(t))-f(\gamma(t_0))}{t-t_0}$$ $$=
\lim_{t\rightarrow t_0}\dfrac{f(\gamma(t))-f(\gamma(t_0))}{\gamma(t)-\gamma(t_0)}\dfrac{\gamma(t)-\gamma(t_0)}{t-t_0}
=\dfrac{{d f}}{dz}(\gamma(t))\gamma^{'}(t)$$ and
thus $ g\in C^1(X) $ and $P_{1,1}(z)=z$, $z\in \mathbb{C}$.

If the result holds for $1\leq i< k$, then we will prove that it also holds for 
$i+1\leq k$.
Using our induction hypothesis,

$$g^{(i)}(t)=\sum\limits_{j=1}^i \dfrac{{d^j f}}{dz^j}(\gamma(t))P_{j,i}(\gamma^{'}(t),
\gamma^{''}(t),\dots,\gamma^{(i)}(t)).$$

We differentiate with respect to $t\in X$. Then 
$$g^{(i+1)}(t)=\sum\limits_{j=1}^i ( \dfrac{{d^{j+1} f}}{dz^{j+1}}(\gamma(t))(P_{j,i}
(\gamma^{'}(t),\gamma^{''}(t),\dots,\gamma^{(i)}(t))\gamma^{'}(t)
$$
$$+\dfrac{{d^j f}}{dz^j}(\gamma(t))\sum\limits_{s=1}^i{\dfrac{\partial P_{j,i}}{\partial 
z_s}}(\gamma^{'}(t),\gamma^{''}(t),\dots,\gamma^{(i)}(t))\gamma^{(s+1)}(t) ) =$$
$$=\dfrac{df}{dz}(\gamma(t))\sum\limits_{s=1}^i{\dfrac{\partial P_{1,i}}{\partial 
z_s}}(\gamma^{'}(t),\gamma^{''}(t),\dots,\gamma^{(i)}(t))\gamma^{s+1}(t)+$$
$$+\sum\limits_{j=2}^i \dfrac{{d^j f}}{dz^j}(\gamma(t))(P_{j-1,i}(\gamma^{'}(t),
\gamma^{''}(t),\dots,\gamma^{(i)}(t))\gamma^{'}(t)+$$ 
$$\sum\limits_{s=1}^i{\dfrac{\partial P_{j,i}}{\partial z_s}}
(\gamma^{'}(t),\gamma^{''}(t),\dots,\gamma^{(i)}(t))\gamma^{(s+1)}(t)))+$$
$$+\dfrac{{d^{i+1} f}}{dz^{i+1}}(\gamma(t))(P_{i,i}
(\gamma^{'}(t),\gamma^{''}(t),\dots,\gamma^{(i)}(t))\gamma^{'}(t)).$$
Thus we obtain $ g\in C^{i+1}(X) $, $$P_{1,i+1}(z_1,z_2,\dots,z_{i+1})=\sum\limits_{s=1}^i{\dfrac{\partial P_{1,i}}{\partial 
z_s}}(z_1,z_2,\dots,z_{i})z_{s+1},$$
$$P_{j,i+1}(z_1,z_2,\dots,z_{i+1})=\sum\limits_{s=1}^i{\dfrac{\partial P_{j,i}}{\partial 
z_s}}(z_1,z_2,\dots,z_{i})z_{s+1}+$$ 
$$P_{j-1,i}(z_1,z_2,\dots,z_{i})z_1,$$ for $j=2,3,\dots,i$, and
$$P_{i+1,i+1}(z_1,z_2,\dots,z_{i+1})=P_{i,i}(z_1,z_2,\dots,z_{i})z_1.$$
Therefore the result holds also for $i+1$ and the proof is complete.
\end{proof}

\begin{prop}
Let $\gamma : X \rightarrow \mathbb{C}$ is a homeomorphism onto $\gamma^*$ with $ \gamma'(t)\neq 0 $ and $ k\in \{1,2,3, 
\cdots\}\cup \{\infty\} $. If $ \gamma \in C^k(X) $, then $ C^{k}(X)=C^{k}(\gamma^*)\circ \gamma $.
\end{prop}

\begin{proof}
By Lemma $2.15$ we have that $ C^{k}(X)\supset C^{k}(\gamma^*)\circ \gamma $. We proceed by proving inductively that $ C^{k}(X)\subset C^{k}(\gamma^*)\circ \gamma $. 

For $ k=1 $ whenever $ f \in C^{1}(X) $ we denote $ g=f\circ \gamma^{-1} $. Then $  g\circ \gamma=f  $ and $$\dfrac{{d g}}{dz}(\gamma(t_0))=\lim_{t\rightarrow 
t_0}\dfrac{g(\gamma(t))-g(\gamma(t_0))}{\gamma(t)-\gamma(t_0)}$$ $$=
\lim_{t\rightarrow t_0}\dfrac{f(t)-f(t_0)}{t-t_0}\dfrac{1}{\dfrac{\gamma(t)-\gamma(t_0)}{t-t_0}} =\dfrac{f'(t_0)}{\gamma^{'}(t_0)}$$ for $ t_0 \in X $. So
the function $ \dfrac{{d g}}{dz} $ exists and is continuous and hence $g \in C^{1}(\gamma^*) $.

Assume that the assertion holds for some $ k $. Given some $f\in C^{k+1}(X) $ we denote $ g=f\circ \gamma^{-1} $. Notice that $  g\circ \gamma=f  $ and $$\dfrac{{d g}}{dz}(\gamma(t_0))=\dfrac{(f'\circ \gamma^{-1})(\gamma(t_0))}{\gamma^{'}(t_0)}$$ for $ t_0 \in X $. From the induction hypothesis we obtain that $ f'\circ \gamma^{-1} \in C^{k}(\gamma^*)$. It is also true that $\gamma' \in C^{k}(X) $. It follows immediately that $ \dfrac{{d g}}{dz} \in C^{k}(\gamma^*)$ or equivalently $ g \in C^{k+1}(\gamma^*)$ and the proof is complete for every finite $ k $.\\
The case $ k=\infty $ follows from the definition of $ C^{\infty}(X)$ and the previous result for  finite $ k $ .
\end{proof}

As we have shown above if $\gamma$ is a homeomorphism defined on $X$, then 
$C(X)=C(\gamma^*)\circ\gamma$. If also $\gamma\in C^k(X)$, where 
$k\in\{0,1,2,\dots\}\cup\{\infty\}$, 
and $\gamma^{'}(t)\neq 0$ for every $t\in X$, then $C^k(X)=C^k(\gamma^*)\circ\gamma$.
This implies that the spaces $C^k(\gamma^*)$ and $C^k(\gamma)$ contain exactly the same
elements. Now we will prove that they also share the same topology.

\begin{prop}
Let $\gamma$ be a homeomorphism defined on $X$ of class
$C^k(X)$ and $\gamma^{'}(t)\neq 0$ for every $t\in X$.
Then the spaces $C^k(\gamma)$ and $C^k(\gamma^*)$ share the same topology.
\end{prop}

\begin{proof}
At first we will prove the proposition in the special case where $X$ is a compact interval 
$I\subset \mathbb{R}$ or the unit circle $T$ and $k\neq \infty$.
In order to do so we will find $0<M,N< \infty$ such that $d_1(f_1,f_2)\leq M 
d_2(f_1,f_2)$ and $d_2(f_1,f_2)\leq N d_1(f_1,f_2)$ for every $f_1, f_2 \in C^k(\gamma)=
C^k(\gamma^*)$,
where $d_1,d_2$ are the metrics of $C^k(\gamma)$ and $C^k(\gamma^*)$, respectively.

Let $f_1, f_2 \in C^k(\gamma^*)$ and
$g_1=f_1\circ\gamma, g_2=f_2\circ\gamma \in C^k(X)$. We notice that
$$\sup\limits_{t\in X}|g_1(t)-g_2(t)|=\sup\limits_{z\in \gamma^*}
|f_1(z)-f_2(z)|\leq d_2(f_1,f_2).$$
In addition, 
\begin{equation}{(g_1-g_2)}^{(i)}(t)=\sum\limits_{j=1}^i \dfrac{{d^j(f_1-f_2)}}{dz^j}
(\gamma(t))P_{j,i}
(\gamma^{'}(t),
\gamma^{''}(t),\dots,\gamma^{(i)}(t)),
\end{equation} for $1\leq j\leq i \leq k$, where $P_{j,i}$ are the
polynomials of Lemma $2.15$. If 
$m_{j,i}=\sup\limits_{t\in X}| P_{j,i}((\gamma^{'}(t),
\gamma^{''}(t),\dots,\gamma^{(i)}(t))|$ and $M_{i}=\max\{m_{j,i}, 1\leq j \leq i\}$, then
$$\sup\limits_{t\in X}|{(g_1-g_2)}^{(i)}(t)|\leq \sum\limits_{j=1}^i 
\sup\limits_{z\in \gamma^*} |\dfrac{{d^j(f_1-f_2)}}{dz^j}(z)|m_{j,i}$$ $$\leq M_{i}\sum
\limits_{j=1}^i\sup\limits_{z
\in \gamma^*} |\dfrac{{d^j(f_1-f_2)}}{dz^j}(z)|\leq M_{i}d_2(f_1,f_2).$$
Consequently $$d_1(f_1,f_2)\leq M d_2(f_1,f_2),$$ where $M=1+M_1+M_2+\dots+M_k$.

We also notice that
$$\sup\limits_{z\in \gamma^*}|f_1(z)-f_2(z)|=\sup
\limits_{t\in X}|g_1(t)-g_2(t)|\leq d_1(f_1,f_2).$$
We will prove inductively that $$\sup\limits_{z\in \gamma^*}
\left|\dfrac{d^i(f_1-f_2)}{dz^i}(z)\right|\leq N_id_1(f_1,f_2),$$
for some $N_i>0$, $i=1,2,\dots$, $i\leq k$.
Since $$P_{1,1}(z_1)=z_1$$ and
$$P_{i+1,i+1}(z_1,z_2,\dots,z_{i+1})=P_{i,i}(z_1,z_2,\dots,z_{i})z_1,$$
it is easy to see that $P_{i,i}(z_1,z_2,\dots,z_{i})={z_1}^i$. Thus
$$s_i=\inf\limits_{t\in X}|P_{i,i}((\gamma^{'}(t),
\gamma^{''}(t),\dots,\gamma^{(i)}(t))|>0,$$ since $\gamma'(t)\neq 0$ for every 
$t\in X$. For $i=1\leq k$ we have $$\sup\limits_{z\in \gamma^*}
\left|\dfrac{d(f_1-f_2)}{dz}(z)\right|\leq \dfrac{1}{s_1}\sup\limits_{t\in X}|(g_1-g_2)'(t)|
\leq \dfrac{1}{s_1}d_1(f_1,f_2),$$ and thus $N_1 =\dfrac{1}{s_1}$. 
Assume that the assertion holds for every $1\leq j\leq i<k$.
Using our induction hypothesis and equation $(1)$ we find 
$$\sup\limits_{t\in X}|{(g_1-g_2)}^{(i+1)}(t)|\geq \sup\limits_{z\in\gamma^*}\left|
\dfrac{d^{i+1}(f_1-f_2)}{dz^{i+1}}(z)\right|s_{i+1}$$ 
$$-\sum\limits_{j=2}^i \sup\limits_{z\in\gamma^*}\left|
\dfrac{{d^j(f_1-f_2)}}{dz^j}(z)\right| m_{j,i+1} \geq\sup\limits_{z\in\gamma^*}\left|
\dfrac{d^{i+1}(f_1-f_2)}{dz^{i+1}}(z)\right|s_{i+1}$$ 
$$-\sum\limits_{j=2}^i N_{j}m_{j,i+1}d_1(f_1,f_2).$$ 
Therefore
$$\sup\limits_{z\in\gamma^*}\left|\dfrac{d^{i+1}(f_1-f_2)}{dz^{i+1}}(z)\right|\leq
N_{i+1}d_1(f_1,f_2),$$ 
where $N_{i+1}=\dfrac{1}{s_{i+1}}+\sum\limits_{j=2}^i \dfrac{N_{j}m_{j,i+1}}{s_{i+1}}$
and the result holds also for $i+1$. The induction is complete.
Now it is easy to deduce that $$d_2(f_1,f_2)\leq N d_1(f_1,f_2),$$
where $N=1+N_1+N_2+\dots+N_k$.

It follows immediately from the above that even in the case where $X$ is any type of interval 
$I\subset\mathbb{R}$ and/or $k=\infty$ the respective topologies of 
the spaces $C^k(\gamma)$ and $C^k(\gamma^*)$ are the same. The basic open subsets of 
$C^k(\gamma)$ are defined by a compact subset of $X$, an $l\in \{0,1,2,\dots\}$, $l\leq k$, 
a function $f\in C^k(\gamma)$ and an $\varepsilon>0$. But if we recall the definition of the
topology of $C^k(\gamma^*)$ and use the above we realize that this basic open subset of $C^k(\gamma)$ is also an open subset of 
$C^k(\gamma^*)$. Similarly every basic open subset of $C^k(\gamma^*)$ is an open 
subset of $C^k(\gamma)$. The proof is complete.
\end{proof}

Combining Propositions $2.14$, $2.16$ and $2.17$ we obtain the following theorem.

\begin{theorem}
Let $\gamma : X \rightarrow \mathbb{C}$ is a homeomorphism onto $\gamma^*$ with $ \gamma'(t)\neq 0 $ for every $ 
t \in X $ and 
$ k\in \{1,2,3, \cdots\}\cup \{\infty\} $. Then $ C^{k}(X)=C^{k}(\gamma^*)\circ \gamma $ if and 
only if $ \gamma \in C^k(X) $. In addition the spaces $C^{k}(\gamma)$ and $ C^{k}(\gamma^*) $ share the same topology.
\end{theorem}

\begin{remark}
One can prove a slightly stronger statement than that of Theorem $2.18$. We do not need to 
assume $ \gamma'(t)\neq 0 $. Then $ C^{k}(X)=C^{k}(\gamma^*)\circ \gamma $ if and only if 
$\gamma \in C^k(X) $ and $ \gamma'(t)\neq 0 $ for every $ t \in X $.
\end{remark}

\begin{remark}
With the definition of the derivative as in the Remark $2.12$ we can define the spaces $ C^k(E) $ for more general sets $ E\subset \mathbb{C} $ but it may 
occur that the space $ C^k(E) $ is not complete.
\end{remark}

Theorem $2.18$ shows that if $ \gamma $ is a homeomorphism and $ \gamma'(t)\neq 0 $ for $ t \in X
$, then $ C^k({\gamma})\approx C^k(\gamma^*) $. Therefore 
we obtain the following corollary.

\begin{corol}
Let $\gamma : X \rightarrow \mathbb{C}$ is a homeomorphism onto $\gamma^*$ with $ \gamma'(t)\neq 0 $ for every $ t\in X$ and $ k\in \{1,2,3, \dots\}\cup \{\infty\} $. Then the space $ C^k({\gamma}) $ is independent of the parametrization of $ \gamma $ and coincides with the space $ C^k(\gamma^*) $.
\end{corol}

\section{Continuous analytic capacities}
In this section we present a few facts for the notion of continuous analytic capacity (\cite{[7]}) that we will need in sections $5$ and $7$ below. Section $7$ leads us to 
generalise this notion and thus define the notion of $p$-continuous analytic capacity.

\begin{defi}
Let $ U\subset\mathbb C $ be open. A function $ f $ belongs to the class $ A(U) $ if $ f\in H(U) $ and $ f $ has a continuous extension on $ \overline{U} $, 
where the closure of $ U $ is taken in $ \mathbb C $. 
\end{defi}

\begin{defi}
Let  $ \Omega $ be the complement in $ \mathbb{C} $ of a compact set. A function $ f $ belongs to the class $ A(\Omega\cup\{\infty\}) $ if $ f\in 
H(\Omega)\cap C(\Omega\cup\{\infty\}) $ and $ f $ has a continuous extension on $ \overline{\Omega} $, where the closure of $ \Omega $ is taken in 
$\mathbb{C}$. 
\end{defi}

By Tietze's extension theorem the extensions in both previous definitions can be considered as extensions on the whole of $ \mathbb C\cup\{\infty\} $ without 
increase of the original norm $ \|f\|_\infty $.

\begin{defi}
Let $ L $ be a compact subset of $ \mathbb{C} $. Let also $ \Omega=\mathbb{C}\setminus L $. We 
denote 
$$a(L)=\sup\{|\lim_{z\rightarrow \infty}z(f(z)-f(\infty))|: f\in A(\Omega\cup\{\infty\}), \| f
\|_{\infty}\leq 1\}$$
the continuous analytic capacity of $ L $.
\end{defi}

It is well known (\cite{[7]}) that $ a(L)=0 $ if and only if $ A(\Omega\cup\{\infty\}) $ 
contains only constant functions.

\begin{theorem}
Let $ L $ be a compact subset of $ \mathbb{C} $ and $ U \subset \mathbb{C} $ be open with $ L\subset U $. Then $ a(L)=0 $ if and only if every $f \in 
C(U)\cap H(U\setminus L) $ belongs to $ H(U) $.
\end{theorem}

\begin{proof}
Assume that every $f \in C(U)\cap H(U\setminus L) $ belongs to $ H(U) $.\\
We consider an arbitrary $ f\in A(\Omega\cup\{\infty\}) $. Since $ f $ can be continuously extended over $ L $, it belongs to $ C(U)\cap H(U\setminus L) $ and 
thus to $ H(U)$. Therefore $ f $ is analytic in $ \mathbb C $ and continuous in $ \mathbb C\cup\{\infty\} $ and hence it is constant.\\
Thus $ a(L)=0 $.\\
Now assume $ a(L)=0 $ and we consider any $f \in C(U)\cap H(U\setminus L) $.\\
There exist two closed curves $ \gamma_1 $ and $\gamma_2$ in $ U $ so that $ \gamma_1 $ surrounds $ L $ and $ \gamma_2$ surrounds $ \gamma_1 $. We define the 
analytic functions
$$ \phi_1(z)=\frac 1{2\pi i}\int_{\gamma_1}\frac{f(\zeta)}{\zeta-z}\,d\zeta\qquad\text{for}\,\, z\,\,\text{in the exterior of}\,\,\gamma_1 $$ 
and
$$ \phi_2(z)=\frac 1{2\pi i}\int_{\gamma_2}\frac{f(\zeta)}{\zeta-z}\,d\zeta\qquad\text{for}\,\, z\,\,\text{in the interior of}\,\,\gamma_2.$$
Then the function $ g $ which equals $ \phi_2-f $ in the interior of $ \gamma_2 $ and $ \phi_1 $ in the exterior of $ \gamma_1 $ is well defined and belongs to
$ A(\Omega\cup\{\infty\}) $. Therefore $ g $ is constant and thus $ f $ is analytic in the interior of $ \gamma_2 $. Hence $ f\in H(U) $.
\end{proof}

Due to the local nature of the proof of the next theorem we shall state a few facts about the 
so-called Vitushkin's localization operator (\cite{[8]}).

Let $ U\subset\mathbb C $ be open and $ f\in C(\mathbb C)\cap H(U) $. Let also $ g\in C^1(\mathbb C) $ have compact support. We define the function
\begin{equation}
\begin{split}
G(z)&=\frac 1{\pi}\iint\frac{f(z)-f(w)}{z-w}\,\frac{\partial g}{\partial \overline w}(w)\,dm(w)\\
&=f(z)g(z)-\frac 1{\pi}\iint\frac{f(w)}{z-w}\,\frac{\partial g}{\partial \overline w}(w)\,dm(w).
\end{split}
\end{equation}
The function $ G $ is continuous in $ \mathbb C\cup\{\infty\} $ with $ G(\infty)=0 $, analytic in $ U\cup(\mathbb C\setminus\supp g)$ and $f-G$ is analytic in 
the interior of the set $ g^{-1}(\{1\}) $.

\begin{defi}
Let $ L $ be a closed subset of $ \mathbb C $. We define 
$$a(L)=\sup\{ a(M) : M\,\,\text{compact subset of}\,\,L\}$$
the continuous analytic capacity of $ L $.
\end{defi}

\begin{theorem}
Let $ L $ be a closed subset of $ \mathbb{C} $. Then $ a(L)=0  $ if and only if for every open set $  U\subset \mathbb{C} $ every $f \in C(U)\cap H(U\setminus 
L) $ belongs to $ H(U) $.
\end{theorem}

\begin{proof}
One direction is immediate from Theorem $3.4$ and Definition $3.5$ and hence we assume that $ a(L)=0 $.\\
We consider an arbitrary open set $  U\subset \mathbb{C} $ which intersects $ L $ and an arbitrary $f \in C(U)\cap H(U\setminus L) $ and we shall prove that 
$ f $ extends analytically over $U\cap L $.\\
Now $ L $ may not be contained in $ U $ but since analyticity is a local property we shall employ Vitushkin's localization operator.\\
Let $ z_0\in U\cap L $ and $ \overline{D(z_0,3\delta)}\subset U $. We consider $ g\in C^1(\mathbb C) $ with $ \supp g\subset \overline{D(z_0,2\delta)} $ such 
that $ g=1 $ in $ \overline{D(z_0,\delta)} $.\\
We also consider the restriction $ F $ of $ f $ in $ \overline{D(z_0,3\delta)} $ and we extend $ F $ so that it is continuous in $ \mathbb C\cup\{\infty\} $.\\
We define as in $(2)$ the function
$$ G(z)=\frac 1{\pi}\iint\frac{F(z)-F(w)}{z-w}\,\frac{\partial g}{\partial \overline w}(w)\,dm(w).$$
Now $G$ is continuous in $ \mathbb C\cup\{\infty\} $ with $ G(\infty)=0 $, analytic in $ (D(z_0,3\delta)\setminus L)\cup(\mathbb C\setminus 
D(z_0,2\delta))=\mathbb C\setminus(\overline{D(z_0,2\delta)}\cap L)$ and $f-G=F-G$ is analytic in $ D(z_0,\delta) $. \\
Since $ a(L)=0 $, we have $ a(\overline{D(z_0,2\delta)}\cap L)=0$ and hence $ G $ is constant $ 0 $ in $ \mathbb C $. Therefore $ f $ is analytic in $ 
D(z_0,\delta) $.\\
Since $ z_0\in U\cap L $ is arbitrary we conclude that $ f\in H(U) $.
\end{proof}

\begin{theorem}
(\cite{[7]}) If $ L $ is a Jordan arc with locally finite length, then $ a(L)=0 $. The same holds for any countable union of such curves. Therefore line segments, 
circular arcs, analytic curves and boundaries of convex sets 
are all of zero continuous analytic capacity.
\end{theorem}

\begin{defi}
Let $ U $ be an open subset of $ \mathbb C $ and $ p \in \{0,1,2, \cdots \}\cup\{\infty\} $. A 
function $ f $ belongs to the class $ A^p(U) $ if $ f\in H(U) $ 
and all derivatives $ f^{(j)} $, $j\in \{0,1,2,\cdots\}$, $ 0\leq j\leq p $, have continuous 
extensions $ f^{(j)} $ on $ \overline{U} $, where the closure of $ U $ is taken in 
$ \mathbb C $. 
\end{defi}

\begin{defi}
Let $ \Omega $ be the complement in $ \mathbb{C} $ of a compact set and $ p \in \{0,1,2, 
\cdots \}\cup\{\infty\} $. A function $ f $ belongs to the class 
$ A^{p}(\Omega\cup\{\infty\}) $ if $ f\in H(\Omega)\cap C(\Omega\cup\{\infty\}) $ and all 
derivatives $ f^{(j)} $, $j\in \{0,1,2,\cdots\}$, $ 0\leq j\leq p $, have continuous 
extensions $ f^{(j)} $ on $ \overline{\Omega} $, where the closure of $ \Omega $ is taken in 
$\mathbb{C}$. 
\end{defi}

\begin{defi}
Let $ L $ be a compact subset of $ \mathbb{C} $ and $ p\in \{0, 1,2, \cdots\}\cup \{\infty\} $. 
Let also $ \Omega=\mathbb{C}\setminus L $. For $p\neq \infty$ we denote the 
$p$-continuous analytic capacity as $$a_p(L)=\sup\{|\lim_{z\rightarrow\infty}z(f(z)-f(\infty))|: f\in A^p(\Omega\cup\{\infty\}), \max_{0\leq j\leq p}\|f^{(j)}\|_{\infty}\leq 1\}.$$
Obviously, $a_0(L)= a(L) $.\\
For $p=\infty$,
$$a_\infty(L)=\sup\{|\lim_{z\rightarrow\infty}z(f(z)-f(\infty))|: f\in 
A^\infty(\Omega\cup\{\infty\}), d(f,0)\leq 1\},$$
where the Fr{\'e}chet distance $d(f,0)$ is defined by
$$d(f,0)=\sum\limits_{j=0}^\infty 2^{-j}\dfrac{\|f^{(j)}\|_\infty}{1+\|f^{(j)}\|_\infty}$$
If $ L $ is a closed subset of $ \mathbb C $ then we define 
$$a_p(L)=\sup\{ a_p(M) : M\,\,\text{compact subset of}\,\,L\}.$$
\end{defi}

\begin{defi}
Let $ L $ be a compact subset of $ \mathbb{C} $ and $ p\in \{0, 1,2, \cdots\}\cup \{\infty\} $. 
Let also $ \Omega=\mathbb{C}\setminus L $. We denote 
$$a'_p(L)=\sup\{|\lim_{z\rightarrow\infty}z(f(z)-f(\infty))|: f\in A^p(\Omega\cup\{\infty\}), 
\|f\|_{\infty}\leq 1\}.$$
Obviously, $a'_0(L)= a(L) $.\\
If $ L $ is a closed subset of $ \mathbb C $ then we define 
$$a'_p(L)=\sup\{ a'_p(M) : M\,\,\text{compact subset of}\,\,L\}.$$
\end{defi}

\noindent It is obvious that $ a_p(L) $ and $ a'_p(L) $ are decreasing functions of $p$.\\

\noindent The following theorem corresponds to Theorem $3.4$.  

\begin{theorem}
Let $ L $ be a compact subset of $ \mathbb{C} $ and $ U \subset \mathbb{C} $ be open with $ L\subset U $ and let $ p \in \{0,1,2, \cdots \}\cup\{\infty\} $. Then $ a_p(L)=0 $ if and only if every $f \in A^p(U\setminus L) $ has an extension to $ A^p(U) $, if and only if $a'_p(L)=0 $.
\end{theorem}

The proof is a repetition of the proof of Theorem $3.4$.\\

\begin{defi}
Let $ L $ be a compact subset of $ \mathbb{C} $. For $ p \in \{0,1,2, \cdots\}\cup \{\infty\} $ we define $b_p(L)$ such that $b_p(L)=0$ when $ a_p(L)=0 $ and $ b_p(L)=\infty $ when $ a_p(L)\neq 0 $. For $ p, q \in \{0,1,2, \cdots\}\cup \{\infty\} $ we will say that $ a_p(L) $ and $a_q(L)$ are essentially different if $ b_p(L)\neq b_q(L)$.
\end{defi}

\begin{defi}
Let $ U $ be an open subset of $ \mathbb C $ and $ p \in \{1,2, \cdots \}\cup\{\infty\} $. A 
function $ f $ belongs to the class $ \tilde A^p(U) $ if $ f\in 
A^p(U) $ and for $ 0\leq j\leq j'\leq p $ the following is true for $ z,w\in\overline U $:
\begin{equation}
f^{(j)}(w)-\sum_{k=0}^{j'-j}\frac 1{k!}f^{(j+k)}(z)(w-z)^k=o(|w-z|^{j'-j})\qquad\text{as}\,\, w\to z.
\end{equation}
This is supposed to hold uniformly for $ z,w $ in compact subsets of $ \overline U $.\\ 
Analogously, if $ \Omega $ is the complement in $ \mathbb{C} $ of a compact set, then a function $ f $ belongs to the class 
$ \tilde A^{p}(\Omega\cup\{\infty\}) $ if $ f\in A^{p}(\Omega\cup\{\infty\}) $ and $(3)$ holds for $ z,w\in\overline \Omega $.
\end{defi}

Note that, since $ f\in H(U) $, relation $(3)$ is automatically true for $ z\in U $ and thus the 
''point'' of the definition is when $ z\in\partial U $.

If $ p $ is finite, then $ f\in\tilde A^p(U) $ admits as a norm $ \|f\|_{\tilde A^p(U)} $ the 
smallest $ M $ such that
$$|f^{(j)}(z)|\leq M\qquad\text{for}\,\,z\in\overline U, 0\leq j\leq p,$$
\begin{equation}
\begin{split}
\Big|f^{(j)}(w)-\sum_{k=0}^{j'-j}\frac 1{k!}f^{(j+k)}(z)(w-z)^k\Big|\leq M|w-z|^{j'-j}\quad& 
w,z\in\overline U, |w-z|\leq 1,\notag\\
&0\leq j\leq j'\leq p.
\end{split}
\end{equation}

It is easy to see that $ \tilde A^p(U) $ with this norm is complete.

If $ p $ is infinite, then using the norms for the finite cases in the standard way, $ \tilde 
A^\infty(U) $ becomes a Fr{\'e}chet space.

There is a fundamental result of Whitney (\cite{[12]}) saying that if $ f\in\tilde A^p(U) $, 
then $ f $ can be extended in $ \mathbb C $ in such a way that the 
extended $ f $ belongs to $ C^p(\mathbb C) $ and that the partial derivatives of $ f $ of order 
$ \leq p $ in $ \mathbb C $ are extensions of the original 
partial derivatives of $ f $ in $ U $.    

\begin{defi}
Let $ L $ be a compact subset of $ \mathbb{C} $ and $ p\in \{0, 1,2, \cdots\}\cup \{\infty\} $. 
Let also $ \Omega=\mathbb{C}\setminus L $. For $p\neq \infty$, we denote
$$\tilde a_p(L)=\sup\{|\lim_{z\rightarrow\infty}z(f(z)-f(\infty))|: f\in\tilde 
A^p(\Omega\cup\{\infty\}), \|f\|_{\tilde A^p(U)}\leq 1\}.$$
In the case $ p=\infty $ the norm $ \|f\|_{\tilde A^p(U)} $ is replaced by the distance from $ f 
$ to $ 0 $ in the metric space structure of $ \tilde A^\infty(U) $.\\
Obviously, $ \tilde a_0(L)= a(L) $.\\
If $ L $ is a closed subset of $ \mathbb C $ then we define 
$$\tilde a_p(L)=\sup\{ \tilde a_p(M) : M\,\,\text{compact subset of}\,\,L\}.$$
\end{defi}

It turns out that in the case $ p\geq 1 $ there is a simple topological characterization of the compact sets $ L $ with $ \tilde a_p(L)=0 $.
\begin{theorem}
Let $ L $ be a compact subset of $ \mathbb{C} $ and $ p\geq 1 $. Then $ \tilde a_p(L)=0 $ if and only if $ L $ has empty interior.
\end{theorem}
\begin{proof}
Let $ L $ have nonempty interior and let the disc $ D(z_0,r_0) $ be contained in $ L $. Obviously, the non-constant function $ \frac 1{z-z_0} $ belongs to $ 
\tilde A^p(\Omega\cup\{\infty\}) $ for all $ p $ and thus $ a_p(L)>0 $ for all $ p $.

Conversely, let $ L $ have empty interior and let $ f $ belong to $ \tilde A^1(\Omega\cup\{\infty\}) $. Then $ f $ is analytic in $ \Omega $ and at every $ 
z\in L $ we have
$$f(w)-f(z)-f'(z)(w-z)=o(|w-z|)\qquad\text{as}\,\,w\to z, w\in\overline\Omega=\mathbb C.$$
Thus $ f $ is analytic at $ z $ (with derivative equal to $ f'(z) $) and hence analytic in all of $ \mathbb C $. Since $ f $ is continuous at $ \infty $, it is 
a constant. Therefore $ \tilde A^1(\Omega\cup\{\infty\}) $ contains only the constant functions and $ \tilde a_1(L)=0 $. 
\end{proof}

\begin{theorem}
There is a compact subset $ L $ of $ \mathbb{C} $ such that $ \tilde a_1(L)=0<\tilde a_0(L) $.
\end{theorem}
\begin{proof}
Due to the last theorem, it is enough to find a compact $ L $ with empty interior and with $ \tilde a_0(L)=a(L)>0 $.

This set $ L $ is a Cantor type set. We consider a sequence $(a_n)$ with $ 0<a_n<\frac 12 $ for every $ n=1,2,3,\ldots $ and construct a sequence $ (L_n) $ of 
decreasing compact sets as follows. $L_0$ is the unit square $ [0,1]\times[0,1] $ and $ L_1 $ is the union of the four squares at the four corners of $ L_0 $ 
with sidelength equal to $a_1$. We then continue inductively. If $ L_n $ is the union of $ 4^n $ squares each of sidelength equal to 
$$l_n=a_1\cdots a_n,$$ 
then each of these squares produces four squares at its four corners each of sidelength equal to $ a_1\cdots a_na_{n+1} $. The union of these new squares is $ 
L_{n+1} $.

We denote $ I_{n,k} $, $ k=1,\ldots,4^n $, the squares whose union is $ L_n $.

Finally, we define
$$L=\bigcap_{n=1}^{+\infty}L_n.$$

It is clear that $ L $ is a totally disconnected compact set. The area of $ L_n $ equals
$$|L_n|=4^n(a_1\cdots a_n)^2=(2a_1\cdots 2a_n)^2.$$

Now we assume that
$$\sum_{n=1}^{+\infty}(1-2a_n)<+\infty.$$

Under this condition we find that the area of $ L $ equals
$$|L|=\lim_{n\to+\infty}|L_n|=\lim_{n\to+\infty}(2a_1\cdots 2a_n)^2>0.$$

Then it is well known (\cite{[7]}) that the function
$$f(z)=\frac 1{\pi}\iint_L\frac 1{z-w}\,dm(w)$$
is continuous in $ \mathbb C\cup\{\infty\} $ with $ f(\infty)=0 $ and holomorphic on $ \mathbb C\setminus L $. Since
$$\lim_{z\to\infty}zf(z)=|L|>0,$$
$ f $ is not identically equal to $ 0 $ and hence $\tilde a_0(L)=a(L)>0$.
\end{proof}

\begin{remark}
The latter part of the above proof shows that if a compact set $ L $ (not necessarily of Cantor type) has strictly positive area, then $ a(L)>0 $, which is a well known fact (\cite{[7]}). 
\end{remark}

The problem of the characterization of the compact sets $ L $ with $ a_p(L)=0 $ seems to be more complicated.

We will show that there is a compact set $ L $ such that $a_0(L)$ and $a_1(L)$ are essentially different; that is $a_0(L)>0$ and $a_1(L)=0$.

\begin{theorem}
There is a compact subset $ L $ of $ \mathbb{C} $ such that $ a_1(L)=0<a_0(L) $.
\end{theorem}
\begin{proof}
We consider the same Cantor type set $ L $ which appeared in the proof of the previous theorem. We keep the same notations.

We now take any $ f $ which belongs to $ A^1(\Omega\cup\{\infty\}) $. Subtracting $ f(\infty) $ from $ f $, we may also assume that $ f(\infty)=0 $.

Let $ z_0\in\Omega $. Then there is $ n_0 $ such that $ z_0\notin L_n $ for all $ n\geq n_0 $.

By Cauchy's formula, for every $ n\geq n_0 $ we have
\begin{equation}
f(z_0)=-\frac 1{2\pi i}\sum_{k=1}^{4^n}\int_{\gamma_{n,k}}\frac{f(z)}{z-z_0}\,dz 
\end{equation}
where $ \gamma_{n,k} $ is the boundary curve of the square $ I_{n,k} $.

Let $ z_{n,k} $ be any point of $ \Omega $ inside $ I_{n,k} $ (for example, the center of the square). It is geometrically obvious that for every $ 
z\in\gamma_{n,k} $ there is a path (consisting of at most two line segments) $ \gamma $ with length $ l(\gamma)\leq 2l_n $ joining $ z $ and $ z_{n,k} $ and 
contained in $ \Omega $ (with the only exception of its endpoint $ z $). Since $ f\in A^1(\Omega\cup\{\infty\}) $, we get
\begin{equation}
\begin{split}
|f(z)-f(z_{n,k})-&f'(z_{n,k})(z-z_{n,k})|=\Big|\int_{\gamma}(f'(\zeta)-f'(z_{n,k}))\,d\zeta\Big|\notag\\
&\leq\int_{\gamma}|f'(\zeta)-f'(z_{n,k})|\,d\zeta\leq\epsilon_nl(\gamma)\leq 2\epsilon_n l_n,
\end{split}
\end{equation}
where $ \epsilon_n\to 0 $ uniformly for $ z\in\gamma_{n,k} $ and for $ k=1,\ldots,4^n $.

Therefore
$$\Big|\int_{\gamma_{n,k}}\frac{f(z)}{z-z_0}\,dz\Big|=\Big|\int_{\gamma_{n,k}}\frac{f(z)-f(z_{n,k})-f'(z_{n,k})(z-z_{n,k})}{z-z_0}\,dz\Big|
\leq \epsilon_n\frac{8l_n^2}{\delta_0}$$
where $ \delta_0 $ is the distance of $ z_0 $ from $ L_{n_0} $.

Thus from $(4)$ we obtain
$$|f(z_0)|\leq\frac{8\cdot 4^nl_n^2}{\pi\delta_0}\,\epsilon_n=\frac{8|L_n|}{\pi\delta_0}\,\epsilon_n\leq\frac 8{\pi\delta_0}\,\epsilon_n.$$
This holds for all $ n\geq n_0 $ and hence $ f(z_0)=0 $ for all $ z_0\in\Omega $.

We proved that the only element $ f $ of $ A^1(\Omega\cup\{\infty\}) $ with $f(\infty)=0 $ is the zero function and thus $ a_1(L)=0 $.
\end{proof}

We will now see a different proof of the above theorem. The proof is longer from the previous one, but it provides a more general result.

\begin{theorem}
Let $K_{1/3}$ be the usual Cantor set lying on $[0,1]$ and $L=K_{1/3}\times K_{1/3}$. Then 
$a_0(L)>0$, but $a_1(L)=0$.
\end{theorem}
\begin{proof}
It is known (\cite{[13]}, \cite{[14]}) that
there exists a function $g$ continuous on $S^2$ and holomorphic off $L$, such that 
$$g'(\infty)=\lim_{z\to\infty}z(g(z)-g(\infty))\neq 0,$$ which implies that $a_0(L)>0$. 

For the second statement we first observe that the area of
$L$ is $0$, as $$L=\bigcap_{n=0}^\infty L_n,$$ where each $L_n$ is the union of $4^n$ squares 
of area $9^{-n}$. 
We will prove that $a_1(L)=0$ or equivalently that every function in 
$A^1(\Omega\cup\{\infty\})$ is entire, where $\Omega=\mathbb{C}\setminus L$. Let  
$f\in A^1(\Omega\cup\{\infty\})$, $\varepsilon>0$ and 
$\varphi_\varepsilon=\varepsilon^{-2}\chi_\varepsilon$,
where $\chi_\varepsilon$ is the characteristic function of the square $S_\varepsilon$ with 
center at $0$ and sides parallel to the axes with length $\varepsilon$. It is easy to see from the continuity of $f$ 
that the convolutions $f*\varphi_\varepsilon$ belong to $C^1(\mathbb{C})$ and converge 
uniformly on $D(0,2)$ to $f$ 
as $\varepsilon\to 0$. Since $f\in A^1(\Omega\cup\{\infty\})$, the partial derivatives 
$\dfrac{\partial u}{\partial x}$, $\dfrac{\partial u}{\partial y}$ and 
$\dfrac{\partial v}{\partial x}$, $\dfrac{\partial v}{\partial y}$ of $u=Ref$ and $v=Imf$, 
respectively, extend continuously on $\mathbb{C\cup\{\infty\}}$ and hence are bounded. 
We will prove that 
\begin{equation}
\dfrac{\partial u*\varphi_\varepsilon}{\partial x}=\dfrac{\partial u}{\partial x}*
\varphi_\varepsilon
\end{equation}
on $\Omega$. Let $(a,b)\in \Omega$ and $h\in\mathbb{R}^*$. Then
$$\dfrac{ (u*\varphi_\varepsilon)(a+h,b)-(u*\varphi_\varepsilon)(a,b)}{h}=$$
$$\varepsilon^{-2}\iint_{S_\varepsilon}\dfrac{u(a+h-x,b-y)-u(a-x,b-y)}{h}dxdy.$$
It is easy to see that for almost every $(x,y)\in S_\varepsilon$ and every $h\in\mathbb{R}$  
the segment $[(a+h-x,b-y),(a-x,b-y)]$ is a subset of $\Omega$. Thus from Mean Value Theorem 
for almost every $(x,y)\in S_\varepsilon$ and every $h\in\mathbb{R}^*$
there exists $q\in\mathbb{R}$, such that
$$\dfrac{u(a+h-x,b-y)-u(a-x,b-y)}{h}=\dfrac{\partial u}{\partial x}(a+q-x,b-y),$$
which remains bounded by a constant $M>0$. For those $(x,y)\in S_\varepsilon$ 
$$\dfrac{u(a+h-x,b-y)-u(a-x,b-y)}{h}$$ converges to
$$\dfrac{\partial u}{\partial x}(a-x,b-y),$$ as $h$ converges to $0$. Using the Dominated Convergence Theorem we obtain $(5)$. Similarly 
$$\dfrac{\partial u*\varphi_\varepsilon}{\partial y}=\dfrac{\partial u}{\partial y}*
\varphi_\varepsilon,
\dfrac{\partial v*\varphi_\varepsilon}{\partial x}=\dfrac{\partial v}{\partial x}*
\varphi_\varepsilon,
\dfrac{\partial v*\varphi_\varepsilon}{\partial y}=\dfrac{\partial v}{\partial y}*
\varphi_\varepsilon.$$
Since the Cauchy-Riemann equations are satisfied for $f$ almost everywhere, we have that
$$\dfrac{\partial u}{\partial x}*\varphi_\varepsilon=
\dfrac{\partial v}{\partial y}*\varphi_\varepsilon$$ 
and
$$\dfrac{\partial u}{\partial y}*\varphi_\varepsilon=
-\dfrac{\partial v}{\partial x}*\varphi_\varepsilon.$$ 
Thus the Cauchy-Riemann equations are satisfied for every $f*\varphi_\varepsilon$ on $\Omega$. Since the interior of $ L $ is void, the set $ \Omega $ is dense in $ \mathbb{C} $.
From the continuity of the partial derivatives of every $f*\varphi_\varepsilon$ on 
$\mathbb{C}$, 
the Cauchy-Riemann equations are satisfied for every $f*\varphi_\varepsilon$ on $\mathbb{C}$, 
which implies that every $f*\varphi_\varepsilon$ is holomorphic on $\mathbb{C}$. Finally, 
the functions $f*\varphi_\varepsilon$ converge uniformly on $D(0,2)$ to $f$, as 
$\varepsilon\to 0$, which combined with Weierstrass theorem implies that $f$ is holomorphic on 
$D(0,2)$ and therefore $f$ is holomorphic on $\mathbb{C}$.
\end{proof}

\begin{remark}
The above proof also shows that if $ L $ is a compact subset of $\mathbb{C}$ of zero area and if for almost every line $ \varepsilon $ which is parallel to the $ x $-axis and for almost every line $ \varepsilon $ which is parallel to the $ y $-axis, $  \varepsilon \cap L = \emptyset $, then $ a_1(L)=0 $. In fact, it suffices that these intersections are finite for a dense set of $ \varepsilon $ parallel to the $ x $-axis and for a dense set of $ \varepsilon $ parallel to the $ y $-axis. 
\end{remark}

\section{Real analyticity on analytic curves}
Let $L\subset \mathbb{C}$ be a closed set without isolated points.
We denote by $C(L)$ the set of continuous functions $f :L \rightarrow \mathbb{C}$. This space endowed with the topology of uniform convergence on the 
compact subsets of $L$ is a complete metric space and thus Baire's theorem is at our 
disposal.

\begin{lemma}
Let $L \subset \mathbb{C}$ be a closed set without isolated points. Let also $ z_0\in L $ be 
the center and $ r>0 $ be the radius of the open disk $D(z_0,r)$ and  
$ 0 < M < +\infty $. The set of continuous functions $ f:L \rightarrow \mathbb{C} $ for 
which there exists a holomorphic on $ D(z_0,r)$ function $ F $ bounded by 
$ M $ such that $ F|_{D(z_0,r)\cap L}=f|_{D(z_0,r)\cap L} $, 
is a closed subset of $ C\left(L\right) $ and has empty interior.
\end{lemma}

\begin{proof}
Let $ A(M,z_0,r) $ be the set of continuous functions $ f:L \rightarrow \mathbb{C} $ for 
which there is a holomorphic function $F$ definedand on $ D(z_0,r)$ and bounded by $M$ such that 
$ F|_{D(z_0,r)\cap L}=f|_{D(z_0,r)\cap L} $. We distinguish two cases according to whether $D(z_0,r)$ is contained or not in $L$.

$1)$ If $D(z_0,r)\subset L$, then the elements of $ A(M,z_0,r) $ belong to $C(L)$, are 
holomorphic on $D(z_0,r)$ and bounded by $M$. Let $ (f_n)_{n\geq 1} $ be a sequence in $ 
A(M,z_0,r) $ converging uniformly on the compact subsets of $L$ to a function $ f $ defined 
on $L$. Then from Weierstrass theorem, it follows that $f$ will be holomorphic on $D(z_0,r)$ and bounded by $M$. Therefore $f\in A(M,z_0,r) $ and $A(M,z_0,r) $ is a closed subset of
$ C\left(L\right) $.

For the second part of the theorem let us assume that $ A(M,z_0,r) $ does not have empty interior. Then there is a function $ f $ in the interior of $ A(M,z_0,r) $, a compact set 
$K\subset L $ and $\delta>0 $ such that 
$$\left\{g \in C\left(L\right): \sup \limits_{z \in K}\left| f(z)-g(z) \right| < 
\delta \right\}\subset A(M,z_0,r).$$ Then the function $h(z)=f(z)+\dfrac{\delta}{2}\bar{z}$, 
$z\in L$ belongs to $ A(M,z_0,r) $ and therefore is holomorphic on $D(z_0,r)$. But then
the function $\dfrac{\delta}{2}\bar{z}$ will be holomorphic on $D(z_0,r)$ which is absurd.
Thus the interior of $ A(M,z_0,r) $ is void.

$2)$ If $D(z_0,r)$ is not contained in $L$, then there is $ w \in D(z_0,r)\backslash L $.
Let $ (f_n)_{n\geq 1} $ be a sequence in $ A(M,z_0,r) $ where $ f_n $ converges uniformly on
compact subsets of $L$ to 
a function $ f $ defined on $L$. Then for $ n=1,2, \dots $ there are holomorphic  
functions $ F_n:D(z_0,r)\rightarrow\mathbb{C} $ bounded by $ M $ such that $ F_n|_{D(z_0,r)
\cap L}=f_n|_{D(z_0,r)\cap L} $. By Montel's theorem, there is a subsequence $ (F_{k_n}) $ 
of $ (F_n) $ which converges uniformly to a function $ F $ on the compact subsets of 
$ D(z_0,r) $  which is holomorphic on $ D(z_0,r) $ and bounded by $ M $. Because 
$ F_{k_n} \rightarrow f$ at $ D(z_0,r)\cap L $ we have that $ F|_{D(z_0,r)\cap L}=f|
_{D(z_0,r)\cap L} $ and so $ f \in  A(M,z_0,r)  $. Therefore $ A(M,z_0,r) $ is a closed 
subset of $ C\left(L\right) $.

If $ A(M,z_0,r) $ does not have empty interior, then there exists a function $ f $ in the interior of 
$ A(M,z_0,r) $, a compact set $K\subset L $ and $ \delta>0 $ such that 
$$\left\{g \in C\left(L\right): \sup \limits_{z \in K}\left| f(z)-g(z) \right| < 
\delta \right\}\subset A(M,z_0,r).$$  We choose $ 0<a<\delta \inf \limits_{z\in K}|z-w| $. 
We notice that this is possible because $\inf \limits_{z\in K}|z-w| >0$, since $w \not\in L$ 
and $K\subset L$. The function $h(z)=f(z)+ \dfrac{a}{2(z-w)} $ 
for $ z\in L $ belongs to $ A(M,z_0,r) $ and therefore it has a holomorphic and bounded 
extension $H$ on $ D(z_0,r)$, such that $ H|_{D(z_0,r)\cap L}=h|_{D(z_0,r)\cap L} $. 
However, there is a holomorphic function $F:D(z_0,r)\rightarrow\mathbb{C}$ which 
coincides with $f$ on $D(z_0,r)\cap L$. By analytic continuation
$H(z)=F(z)+ \dfrac{a}{2(z-w)} $ for $ z\in D(z_0,r)\setminus \{z_0\}$, since they are equal 
on $L\cap (D(z_0,r)\setminus \{z_0\}) $, which contains
infinitely many points close to $z_0$, all of them being non isolated. As a result $H$ is 
not bounded at $D(z_0,r)$ which is a contradiction. Thus $ A(M,z_0,r) $ has empty interior.
\end{proof}

\begin{defi}
Let $L\subset \mathbb{C}$ be a closed set without isolated points and $z_0 \in L$. A 
function $f\in C(L)$ belongs to the class of non-holomorphically extendable at $z_0$ functions if there is no pair of an open disk $D(z_0,r)$, $r>0$ and a holomorphic function $F:D(z_0,r)\rightarrow \mathbb{C}$ such that $F|_{D(z_0,r)\cap L}=f|_{D(z_0,r)\cap L}$. 
\end{defi}

\begin{theorem}
Let $L\subset \mathbb{C}$ be a closed set without isolated points and $z_0 \in L$. The class of non-holomorphically extendable at $z_0$ functions of $C(L)$ is a dense and $ G_\delta $ subset of $C(L)$.
\end{theorem}

\begin{proof}
The set $$ \bigcap \limits^{\infty}_{n=1}\bigcap \limits^{\infty}_{M=1} \left( C(L) \backslash A\left(M,z_0,\dfrac{1}{n}\right) \right) $$ is a dense $ G_\delta $ subset of $ C(L) $ according to Baire's Theorem and coincides with the class 
of non-holomorphically extendable at $z_0$ functions, since every holomorphic function on $D(z_0,r)$ is bounded when restricted on $D(z_0,r')$ for $r'<r$.
\end{proof}

\begin{defi}
Let $L\subset \mathbb{C}$ be a closed set without isolated points. A function $f\in C(L)$ 
belongs to the class of nowhere holomorphically extendable functions defined and continuous 
on $L$, if there exists no pair of an open disk 
$D(z_0,r),z_0 \in L, r>0$ and a holomorphic function $F:D(z_0,r)\rightarrow \mathbb{C}, $ 
such that $F|_{D(z_0,r)\cap L}=f|_{D(z_0,r)\cap L}$. 
\end{defi}

\begin{theorem}
Let $L\subset \mathbb{C}$ be a closed set without isolated points. The class of nowhere 
holomorphically extendable functions of $C(L)$ is a dense and $ G_\delta $ subset of $ C(L) $.
\end{theorem}

\begin{proof}
Let $ z_l\in L, l=1,2,3,\dots $ be a dense sequence. Then the set $$ \bigcap 
\limits^{\infty}_{l=1}\bigcap 
\limits^{\infty}_{n=1}\bigcap \limits^{\infty}_{M=1} \left( C(L) \backslash A\left( M,z_l,
\dfrac{1}
{n}\right) \right) $$ is a dense $ G_\delta $ subset of $ C(L) $ according to Baire's 
Theorem. This set coincides with the class of nowhere holomorphically extendable 
functions of $C(L)$, since every holomorphic function on $D(z_0,r)$ is 
bounded when restricted on $D(z_0,r')$ for $r'<r$.
\end{proof}

The proof of the above results can be used to prove similar results at some special cases.
Let $\gamma : I \rightarrow \mathbb{C}$ be a a continuous and locally injective curve, 
where $I$ is an interval in $\mathbb{R}$ of any type. The symbol $\gamma ^*$ will be used instead of $\gamma(I)$. It is obvious that $\gamma ^*$ has no isolated points. We also recall Definition 
$2.2$ of $C^k(\gamma)$.

\begin{defi}
Let $\gamma : I \rightarrow \mathbb{C}$ be a locally injective curve and 
$z_0=\gamma(t_0),t_0 \in I$. A function $f: \gamma^* \rightarrow \mathbb{C}$ belongs to the class of non-holomorphically extendable at $(t_0,z_0=\gamma(t_0))$ functions, if  
there are no open disk 
$D(z_0,r), r>0$ and $\eta >0$ and a holomorphic function $F:D(z_0,r)\rightarrow 
\mathbb{C}$, such that $\gamma((t_0-\eta,t_0+\eta)\cap I) \subset
D(z_0,r)$ and $ F(\gamma(t))=f(\gamma(t)) $ for all $t\in(t_0-\eta,t_0+\eta)\cap I$. 
Otherwise we say that $f$ is holomorphically extendable at $(t_0,z_0=\gamma(t_0))$.
\end{defi}

\begin{theorem}
Let $k,l\in \left\{ 0,1,2,\dots\right\}\cup\left\{\infty\right\}$ such that $l\leq k$. Let 
also $\gamma : I \rightarrow \mathbb{C}$ be a locally injective function of 
$C^k(I)$ and $t_0 \in I$. The class of non-holomorphically extendable at $(t_0,z_0=\gamma(t_0))$ functions belonging to $C^l(\gamma )$ is a dense and $ G_\delta $ subset of $ C^l(\gamma )$.
\end{theorem}

\begin{proof}
Let $r>0$ and $\eta>0$ such that $\gamma(t_0-\eta,t_0+\eta)\subset D(\gamma(t_0),r)$. Let 
also $ A(M,z_0,r,\eta,l) $ be the set of $C^l(I)$ functions for which 
there is a holomorphic function $F$ defined on $ D(z_0,r)$ and bounded by $M$, such 
that $F(\gamma(t))=f(\gamma(t)) $ for every $t\in(t_0-\eta,t_0+\eta)\cap I$; That is we assume 
that $|F(z)|\leq M$ for all $z\in D(z_0,r)$.

Since $\gamma \in C^k(I)$ the open disk $D(z_0,r)$ is not contained in $\gamma ^*$ (see 
Proposition $6.2$) and thus there is $ w \in D(z_0,r)
\backslash \gamma ^* $. Similarly to the proof of Lemma $4.1$ $A(M,z_0,r,\eta,l) $ is a closed subset of $ C^l(\gamma ) $. 

If $A(M,z_0,r,\eta,l) $ does not have empty interior, then there is a function $ f $ in the 
interior of 
$ A(M,z_0,r,\eta,l) $, $b\in \{0,1,2,\dots\}$, a compact set $K\subset I $ and $ \delta>0 $ 
such that 
\begin{align*}
\{g \in C^k(\gamma ): \sup \limits_{t \in K}\left| (f\circ \gamma)^{(j)}(t)-(g\circ 
\gamma)^{(j)}(t) \right| < 
\delta,\\ 0\leq j \leq b \}\subset A(M,z_0,r,\eta,l).
\end{align*}
We choose $ 0<a<\delta \min\{
\inf \limits_{t\in K}|\gamma(t)-w|,\inf \limits_{t\in K}|\gamma(t)-w|^2,\dots, 
\dfrac{1}{b!}\inf \limits_{t\in K}|\gamma(t)-w|^{b+1} \}.$
This is possible because $w\not\in\gamma^*$ and $\gamma(K)\subset \gamma^*$. The 
function $h(z)=f(z)+ \dfrac{a}{2(z-w)} $ for $ z\in \gamma ^* $ belongs to $ A(M,z_0,r,
\eta,l) $,
since $\gamma \in C^k(I)$. Similarly to the proof of Lemma $4.1$ we are led to a contradiction.
Therefore $A(M,z_0,r,\eta,l) $ has empty interior.

Let $s_{n,m}$, $n=1,2,3,\dots$, $m=1,2,3,\dots$ be a sequence such that $\lim\limits_{m 
\rightarrow \infty}s _{n,m}=0$ for every $n=1,2,3,\dots$ and $\gamma(t_0-s_{n,m},t_0+s_{n,m}) 
\subset D \left( \gamma(t_0),\dfrac{1}{n}\right)$ for every $n=1,2,3,\dots$ and every 
$m=1,2,3,\dots$. Then the class of non-holomorphically 
extendable at $(t_0,z_0=\gamma(t_0))$
functions belonging to $C^l(\gamma )$ coincides with the set 
$$ \bigcap 
\limits^{\infty}_{n=1}\bigcap \limits^{\infty}_{m=1} \bigcap \limits^{\infty}_{M=1} \left( 
C^l(\gamma) \backslash A\left(M,z_0,
\dfrac{1}{n},s_{n,m},l \right) \right) ,$$ because every holomorphic function on $D(z_0,r)$ 
becomes bounded if we restrict it on $D(z_0,r')$ for $r'<r$. Thus according to Baire's 
theorem the class of non-holomorphically extendable at $(t_0,z_0=\gamma(t_0))$ functions 
of $C^l(\gamma )$ is a dense and $ G_\delta $ subset of $ C^l(\gamma ) $.
\end{proof}

\begin{defi}
Let $\gamma : I \rightarrow \mathbb{C}$ be a locally injective curve. A function $f :\gamma^* \rightarrow \mathbb{C} $ belongs to the class of
nowhere holomorphically extendable functions if  
there are no open disk 
$D(z_0,r),z_0=\gamma(t_0),t_0\in I, r>0$ and $\eta >0$ and a holomorphic function 
$F:D(z_0,r)\rightarrow 
\mathbb{C},$ such that $\gamma((t_0-\eta,t_0+\eta)\cap I) \subset D(z_0,r)$ and $ F(\gamma(t))=f(\gamma(t)) $ for every $t\in(t_0-\eta,t_0+\eta)\cap I$.
\end{defi}

\begin{theorem}
Let $k,l\in \left\{ 0,1,2,\dots\right\}\cup\left\{\infty\right\}$ such that $l\leq k$. Let 
also $\gamma : I \rightarrow \mathbb{C}$ be a locally injective function of
$C^k(I)$. The class of nowhere holomorphically extendable 
functions of $C^l(\gamma )$ is a dense and $ G_\delta $ subset of 
$ C^l(\gamma ) $.
\end{theorem}

\begin{proof}
Let $t_n\in I$, $n=1,2,3,\dots$ be a dense sequence in $I$. Then the class of nowhere 
holomorphically extendable functions of $C^k(\gamma )$ coincides with the 
intersection over 
every $n=1,2,3,\dots$ of the classes of non-holomorphically extendable at $(t_n,z_n=
\gamma(t_n))$ functions
of $C^l(\gamma )$. Since the classes of non-holomorphically extendable at 
$(t_n,z_n=\gamma(t_n))$
functions of $C^l(\gamma )$ are dense and $ G_\delta $ subsets of 
$ C^l(\gamma ) $ according to Theorem $4.7$, it follows that the class of nowhere 
holomorphically extendable functions of $C^l(\gamma )$ is a dense and 
$ G_\delta $ subset of $ C^l(\gamma ) $ from Baire's theorem.
\end{proof}

We intend to prove results about real analyticity using results about non-extendability.
At first we notice that Proposition $2.9$ and Theorem $4.9$ immediately prove the following theorems.

\begin{theorem}
Let $\gamma : I \rightarrow \mathbb{C}$ be an analytic curve and $t_0 \in I$. For $ k=0,1,2, \dots$ or $k=\infty$ the class of functions $ f \in C^{k}(\gamma) $ which are 
not real analytic at $(t_0,z_0=\gamma(t_0))$ is a dense and $ G_{\delta} $ subset of $ C^{k}
(\gamma) $.
\end{theorem}

\begin{theorem}
Let $\gamma : I \rightarrow \mathbb{C}$ be an analytic curve. For $ k=0,1,2, \dots$ or $k=\infty$ the class of functions $ f \in C^{k}(\gamma) $ which are 
nowhere real analytic is a dense and $ G_{\delta} $ subset of $ C^{k}(\gamma) $.
\end{theorem}

\begin{remark}
The fact that the class of functions $ f \in C^{\infty}([0,1]) $ which are 
nowhere real analytic is itself a dense and $ G_{\delta} $ subset of 
$ C^{\infty}([0,1]) $ strengthens the result \cite{[1]}, where it is only proven that 
this class contains a dense and $ G_{\delta} $ subset of $ C^{\infty}([0,1]) $.
\end{remark}

\begin{prop}
Let $ \gamma:I \rightarrow \mathbb{C}$ be an analytic curve and $\gamma^*=\gamma(I)$. Let also 
$\Phi : \gamma^* \rightarrow \mathbb{C}$
be a homeomorphism of $\gamma^*$ on $\Phi(\gamma^*)\subset\mathbb{C}$ and $\delta=\Phi\circ 
\gamma$. Then the set of functions $f\in C^k(\delta), k\in 
\left\{ 0,1,2,\dots\right\}\cup\left\{\infty\right\}$ which are nowhere analytic is a
$G_{\delta} $ and dense subset of $C^k(\delta)$.
\end{prop}

\begin{proof}
The map $S:C^k(\gamma)\rightarrow C^k(\delta)$ defined by $S(g)=g\circ \Phi^{-1}$, $g\in
C^k(\gamma)$ is a surjective isometry. Also a function $g \in C^k(\gamma)$ is nowhere analytic if 
and only if $S(g)$ is nowhere analytic. Theorem $4.11$ combined with the above facts yields 
the result.
\end{proof}

\begin{corol}
Assume that $J$ is an interval and $\gamma(t)=t$ or $J=\mathbb{R}$ and
$\gamma(t)=e^{it}$. Let $X$ denote the image of $\gamma$ and $\Phi : X \rightarrow \mathbb{C}$ be a homeomorphism of $X$ on $\Phi(X)\subset \mathbb{C}$ and $\delta=\Phi\circ \gamma$. Then the
set of functions $f\in C^k(\delta), k\in 
\left\{ 0,1,2,\dots\right\}\cup\left\{\infty\right\}$ which are nowhere analytic is a
$G_{\delta} $ and dense subset of $C^k(\delta)$.
\end{corol}

\begin{proof}
The curve $\gamma$ is an analytic curve defined on an interval. The result follows from
Proposition $4.13$.
\end{proof}

\begin{remark}
According to Corollary $4.14$ for any Jordan curve or Jordan arc $\delta$ with a 
suitable parametrization generically on $C^k(\delta)$, $ k\in 
\left\{ 0,1,2,\dots\right\}\cup\left\{\infty\right\}$ every function is nowhere analytic.
In fact this holds for all parametrizations of $\delta^*$ and the spaces $C^k(\delta)$
are the same for all parametrizations so that $ \delta $ is a homeomorphism between the unit circle $ T $ or $ [0,1] $ and $ \delta^* $ (see Preliminaries).
\end{remark}

\section{Extendability of functions on domains of finite connectivity}

We start this section with the following general fact.

\begin{prop}
Let $n \in\{1,2,\dots\} $. Let also $X_1,\dots ,X_n$ be complete metric spaces and $A_1,\dots ,A_n$ dense and $G_\delta$ subsets of $X_1,\dots,X_n$, respectively. Then the metric space $X_1\times\dots\times X_n$, endowed with the product topology, is complete and 
$A_1\times\dots \times A_n$ is a dense and $G_\delta$ subset of $X_1\times\dots\times X_n$.
\end{prop}

\begin{proof}
Obviously, $X_1\times\dots\times X_n$ is a complete metric space. 
If $A_i= \bigcap\limits_{k=1}^{\infty} A_{i,k}$, where $A_{i,k}$ are dense and
open subsets of $X_i$ for $i=1,\dots,n$, $k=1,2,\dots$, then $A_1\times\dots \times A_n=
\bigcap\limits_{k=1}^{\infty} (A_{1,k}\times\dots\times A_{n,k})$, where
$A_{1,k}\times\dots\times A_{n,k}$ are open and dense subsets of $X_1,\dots,X_n$. 
Baire's theorem completes the proof.
\end{proof}

\begin{remark}
The result of Proposition $5.1$ can easily be extended to infinite denumerable products, but we will not use it in the current paper.
\end{remark}

\begin{defi}
Let $n \in\{1,2,\dots\}$ and $\gamma_i : I_i \rightarrow \mathbb{C}$, $i=1,\dots,n$ be 
continuous and locally injective curves, where $I_i$ are intervals. We define the space
$C^{p_1,\dots,p_n}(\gamma_1,\dots,\gamma_n)=C^{p_1}(\gamma_1)\times\dots\times C^{p_n}(\gamma_n)$,
where $p_i\in\{0,1,2,\dots\}\cup\{\infty\}$ for $i=1,\dots,n$. The space 
$C^{p_1,\dots,p_n}(\gamma_1,\dots,\gamma_n)$ is endowed with the product topology and becomes
a complete metric space.
\end{defi}

We can regard the above space as the class of functions $f$, which are defined on the 
disjoint union $\gamma_1^*\cup\dots\cup\gamma_n^*$ of the locally injective curves $\gamma_1,\dots,\gamma_n$, where $f|_{\gamma_i}$ belongs to $C^{p_i}(\gamma_i)$.

\begin{defi}
Let $n \in\{1,2,\dots\}$ and $\gamma_i : I_i \rightarrow \mathbb{C}$, $i=1,\dots,n$, be 
locally injective curves, where $I_i$ are intervals. 
A function $f$ defined on the disjoint union $\gamma_1^*\cup\dots\cup\gamma_n^* $ for
$z\in \gamma_i^*$
belongs to the class of nowhere holomorphically extendable functions if the restriction of $f$ on $\gamma_i^*$
belongs to the class of nowhere holomorphically extendable functions defined on
$\gamma_i^*$, respectively, for every $i=1,\dots,n$.
\end{defi}

\begin{theorem}
Let $n \in\{1,2,\dots\}$, $p_i,q_i \in\{0,1,\dots\}\cup \{\infty\}$ such that $p_i\leq q_i$ for 
$i=1,\dots,n$ . Let also $\gamma_i:I_i \rightarrow \mathbb{C}$, $i=1,\dots,n$, be locally 
injective functions of $C^{q_i}(I_i)$, where $I_i$ are intervals. The class of 
nowhere homolomorphically extendable functions of $C^{p_1,\dots,p_n}(\gamma_1,\dots,
\gamma_n)$ is a dense and $G_\delta$ subset of $C^{p_1,\dots,p_n}(\gamma_1,\dots,\gamma_n)$.
\end{theorem}

\begin{proof}
Let $A_i$ be the class of nowhere holomorphically extendable functions of 
$C^{p_i}(\gamma_i)$. Then the set $A_1\times\dots\times A_n$ coincides with
the class of nowhere holomorphically extendable functions of $C^{p_1,\dots,p_n}
(\gamma_1,\dots,\gamma_n)$. It follows from Theorem $4.9$ that the sets $A_1,\dots,A_n$ are dense 
and $G_\delta$ subsets of $C^{p_1}(\gamma_1),\dots, C^{p_n}(\gamma_n)$, respectively, which
combined with Proposition $5.1$ implies that the class $A_1\times\dots\times A_n$
is a dense and $G_\delta$ subset of $C^{p_1,\dots,p_n}(\gamma_1,\dots,\gamma_n)$.
\end{proof}

\begin{defi}
Let $n \in\{1,2,\dots\}$. Let also $\gamma_i : I_i \rightarrow \mathbb{C}$, $i=1,\dots,n$, be 
locally injective curves, where $I_i$ are intervals. A function 
$f$ on the disjoint union $\gamma_1^*\cup\dots\cup\gamma_n^* $, $f(z)=f_i(z)$ for $z\in
\gamma_i^*$, is nowhere real analytic if the functions $f_i$ are nowhere real analytic for 
$i=1,\dots,n$.
\end{defi}

The proof of the following theorem is similar to the proof of Theorem $5.5$.

\begin{theorem}
Let $n \in\{1,2,\dots\}$ and $p_i \in\{0,1,\dots\}\cup \{\infty\}$, $i=1,\dots,n$. Let also
$\gamma_i :I_i \rightarrow \mathbb{C}$ be analytic curves, where $I_i$ are intervals,  
$\Phi_i:\gamma_i^* \rightarrow \mathbb{C}$ be homeomorphisms of $\gamma_i^*$ on 
$\Phi_i(\gamma_i)\subset \mathbb{C}$ and let $\delta_i=\Phi_i\circ\gamma_i$, 
$i=1,\dots,n$. 
The class of nowhere analytic functions $f\in C^{p_1,\dots,p_n}(\delta_1,\dots,\delta_n)$ is a 
dense and $G_\delta$ subset of $C^{p_1,\dots,p_n}(\delta_1,\dots,\delta_n)$.
\end{theorem}

From now on, we will consider that $ p_1=p_2= \cdots =p_n $. As we did for the spaces 
$ C^{p_1,\dots,p_n} $, we will prove analogue generic results in the space $A^p(\Omega)$, 
where 
$\Omega$ is a planar domain bounded by the disjoint Jordan curves $\gamma_1,\dots,\gamma_n$,
$n\in\{1,2,\dots\}$. More specifically, we will define the following spaces: 

\begin{defi}
Let $p\in\{0,1,\dots\}\cup\{\infty\}$ and let $\Omega$ be a bounded domain in $\mathbb{C}$.
A function $f$ belongs to the class $A^p(\Omega)$ if it is 
holomorphic on $\Omega$ and every derivative 
$f^{(j)}$ can be continuously extended on $\overline{\Omega}$ for every
$j\in\{0,1,\dots\}$, $j\leq p$. The space $A^p(\Omega)$ is endowed with the topology of 
uniform convergence on $\overline{\Omega}$ of every derivative $f^{(j)}$ for all
$j\in\{0,1,\dots\}$, $j\leq p$ and becomes a complete metric space.
\end{defi}

\begin{remark}
In particular cases it is true that $A^p(\Omega)$ is included in $C^p(\partial\Omega)$ as
a closed subset. We will not examine now under which more general sufficient conditions this
remains true.
\end{remark}

\begin{remark}
If $ \Omega $ is arbitrary open set in $ \mathbb{C} $ (possibly unbounded), then a holomorphic function $f: \Omega \rightarrow \mathbb{C}$ belongs to the 
class $A^p(\Omega)$ if for every $ j\in \{0,1,2, \cdots\}, j\leq p $ the derivative $ f^{(j)} $ has a continuous extension from $ \Omega $ to its closure $ 
\overline{\Omega} $ in $ \mathbb{C} $. The topology of $ A^p(\Omega) $ is defined by the seminorms $ \sup\limits_{z \in \overline{\Omega}, |z|\leq n} 
|f^{(l)}(z)|, l\in \{0,1,2, \cdots\}, l\leq p, n \in \mathbb{N}$.
\end{remark}

\begin{defi}
Let $\Omega$ be a bounded domain in $\mathbb{C}$ defined by a finite number of
disjoint Jordan curves and $z_0 \in \partial \Omega$. A continuous function 
$f: \overline{\Omega}\rightarrow \mathbb{C}$
belongs to the class of non-holomorphically extendable at $z_0$ functions in the sense of
Riemann surfaces if there do not exist open disks $D(z_0,r)$, $r>0$ and $D(z_1,d)$,
$z_1 \in D(z_0,r)\cap \Omega$, $d>0$ such that $D(z_1,d)\subseteq D(z_0,r)\cap \Omega$,
and a holomorphic function $F:D(z_0,r) \rightarrow \mathbb{C}$ such that 
$F|_{D(z_1,d)}=f|_{D(z_1,d)}$.
\end{defi}

\begin{theorem}
Let $p\in \{0,1,\dots\}\cup \{\infty\}$ and $\Omega$ be a bounded domain in $\mathbb{C}$ 
defined by a finite number of
disjoint Jordan curves. Let also $z_0 \in \partial \Omega$. The class of non-holomorphically 
extendable at $z_0$ functions of $A^p(\Omega)$ in the sense of Riemann surfaces is a dense
and $G_\delta$ subset of $A^p(\Omega)$.
\end{theorem} 

\begin{proof}
Let $M>0$, $r>0$, $z_1 \in D(z_0,r)\cap \Omega$ and $d>0$ such that 
$D(z_1,d)\subset D(z_0,r)\cap \Omega$. Let 
also $ A(p,\Omega,z_0,r,z_1,d,M) $ be the set of $A^p(\Omega)$ functions $f$ for which 
there exists a holomorphic function $F$ on $ D(z_0,r)$, such 
that $ |F(z)|\leq M $ for every $z\in D(z_0,r)$ and $F|_{D(z_1,d)}=f|_{D(z_1,d)}$. We will 
first show
that the class $ A(p,\Omega,z_0,r,z_1,d,M) $ is a closed subset of $A^p(\Omega)$ with 
empty interior.

Let $ (f_n)_{n\geq 1} $ be a sequence in $ A(p,\Omega,z_0,r,z_1,d,M) $ converging in the 
topology of $A^p(\Omega)$ to 
a function $ f $ of $A^p(\Omega)$. This implies that $f_n$ converges uniformly on 
$\overline{\Omega}$ to f. Then for $ n=1,2, \dots $ there are holomorphic  
functions $ F_n:D(z_0,r)\rightarrow\mathbb{C} $ bounded by $ M $ such that $ F_n|_{D(z_1,d)
}=f_n|_{D(z_1,d)} $. By Montel's theorem there is a subsequence $ (F_{k_n}) $ 
of $ (F_n) $ which converges uniformly on the compact subsets of
$ D(z_0,r) $ to a function $ F $ which is holomorphic and bounded by $ M $ on $ D(z_0,r) $. 
Since
$ F_{k_n}$ converges to $f$ on $ D(z_1,d)$ we have that $ F|_{D(z_1,d)}=f|
_{D(z_1,d)} $ and so $ f \in A(p,\Omega,z_0,r,z_1,d,M)$. Therefore
$A(p,\Omega,z_0,r,z_1,d,M)$ is a closed subset of $A^p(\Omega)$.

In addition if $A(p,\Omega,z_0,r,z_1,d,M)$ has non-empty interior, then there exist 
$f\in A(p,\Omega,z_0,r,z_1,d,M)$, $l\in\{0,1,2,\dots\}$ and $\epsilon>0$ such that
\begin{align*}
\{g \in A^p(\Omega ): \sup \limits_{z \in \overline{\Omega} }\left| f^{(j)}(z)-g^{(j)}(z) 
\right| < \epsilon,\\ 0\leq j \leq l \}\subset A(p,\Omega,z_0,r,z_1,d,M).
\end{align*} 
We choose 
$w\in D(z_0,r)\setminus\overline{\Omega}$ and 
$ 0<\delta<\epsilon \min\{
\inf \limits_{z\in \overline{\Omega}}|z-w|,\inf \limits_{z\in \overline{\Omega}}|z-w|^2,\dots, 
\dfrac{1}{l!}\inf \limits_{z\in \overline{\Omega}}|z-w|^{l+1} \}.$
This is possible because 
$w \not\in \overline{\Omega}$. The function $h(z)=f(z)+\dfrac{\delta}{2(z-w)} $ 
belongs to $A(p,\Omega,z_0,r,z_1,d,M)$ and has a 
holomorphic and bounded extension $H$ on $ D(z_0,r)$. 
However, there is a holomorphic function $F:D(z_0,r)\rightarrow\mathbb{C}$ which 
coincides with $f$ on $D(z_1,d)$. By analytic continuation
$H(z)=F(z)+ \dfrac{\delta}{2(z-w)} $ for every $ z\in D(z_0,r)\setminus \{z_0\}$, since they are 
equal on $D(z_1,d) $. As a result $H$ is not bounded at $D(z_0,r)$ which yields the desired contradiction. Thus $A(p,\Omega,z_0,r,z_1,d,M)$ has empty interior.

Next let us consider $B$ be the set of $(r,z,d,M)$, where $r=\dfrac{1}{n}$, 
$d=\dfrac{1}{m}$ for some $n,m\in\{1,2,\dots\}$ for which there exists 
$z\in \mathbb{Q}+i\mathbb{Q}$ such that $D(z,d)\subset D(z_0,r)\cap \Omega$, and 
$M\in\{1,2,\dots\}$. Notice that there is a sequence $(b_n)$ such that $B=\{b_n:n\in\{1,2,\dots\}\}$.
Then the class of nowhere holomorphically 
extendable functions of $A^p(\Omega)$ coincides with the set 
$$ \bigcap \limits^{\infty}_{n=1}\left( A^p(\Omega) \backslash A\left(p,\Omega,z_0,b_n
 \right) \right) ,$$ because every holomorphic function on $D(z_0,r)$ 
becomes bounded when restricted on $D(z_0,r')$ for $r'<r$. Thus according to Baire's 
theorem the class of non-holomorphically extendable at $z_0$ functions 
of $A^p(\Omega )$ is a dense and $ G_\delta $ subset of $ A^p(\Omega ) $.
\end{proof}

\begin{defi}
Let $\Omega$ be a bounded domain in $\mathbb{C}$ defined by a finite number of
disjoint Jordan curves. A continuous function $f: \overline{\Omega}\rightarrow \mathbb{C}$ belongs to the class of nowhere holomorphically extendable functions in the sense of
Riemann surfaces if for every $z_0\in\partial \Omega$ $f$
belongs to the class of non-holomorphically extendable at $z_0$ functions in the sense of 
Riemann surfaces.
\end{defi}

\begin{theorem}
Let $p\in \{0,1,\dots\}\cup \{\infty\}$ and let $\Omega$ be a bounded domain in $\mathbb{C}$ 
defined by a finite number of
disjoint Jordan curves. The class of nowhere holomorphically 
extendable functions of $A^p(\Omega)$ in the sense of Riemann surfaces is a dense
and $G_\delta$ subset of $A^p(\Omega)$.
\end{theorem}

\begin{proof}
Let $z_l$, $l=1,\dots$ be a dense sequence of $\partial\Omega$. The class $A(z_l)$ of
non-holomorphically extendable at $z_l$ functions of $A^p(\Omega)$ in 
the sense of Riemann surface is a dense and $G_
\delta$ subset of $A^p(\Omega)$ from Theorem $5.12$. Notice that the set $\bigcap\limits_{l=1}^{\infty}A(z_l)$ coincides 
with the class of nowhere holomorphically extendable functions of $A^p(\Omega)$ in the sense of Riemann surfaces and from Baire's theorem is a dense and $G_\delta$ subset of $A^p(\Omega)$.
\end{proof}

\begin{remark}
In \cite{[5]} it has been also proved that the class of nowhere 
holomorphically extendable functions of $A^\infty(\Omega)$ in the sense of Riemann surfaces 
is a dense and $G_\delta$ subset of $A^\infty(\Omega)$. The method in \cite{[5]} comes from the
theory of Universal Taylor Series and is different from the method in the present paper.
\end{remark}

Now we will examine a different kind of extendability.

\begin{defi}
Let $\Omega$ be a bounded domain in $\mathbb{C}$ defined by a finite number of
disjoint Jordan curves $ \gamma_{1}, \cdots ,\gamma_{n}$. Let also 
$z_0 \in \partial \Omega$. A continuous function $f: \overline{\Omega}\rightarrow \mathbb{C}$ 
belongs to the class of non-holomorphically extendable at $z_0$ functions if there exist no 
pair of an open disk $D(z_0,r)$, $r>0$ and a holomorphic function 
$F :D(z_0,r) \rightarrow \mathbb{C}$ such that $F(z)=f(z)$ for every 
$z\in D(z_0,r)\cap\partial\Omega$. Otherwise we will say that $f$ is holomorphically 
extendable at $z_0$.
\end{defi}

\begin{remark}
If $ \gamma_{i}:I_i\rightarrow \mathbb{C} $, $I_i=[a_i,b_i]$, $a_i<b_i$ is continuous
and $\gamma_i(x)=\gamma_i(y)$ for $x,y\in [a_i,b_i]$ if and only if $x=y$ or 
$x,y\in \{a_i,b_i\}$, then we observe that a 
function $ f $ belongs to the 
class of non-holomorphically extendable at $z_0=\gamma_{i}(t_0), t_0\in I_i$ of Definition 
$5.16$ if and only if there are no open disk 
$D(z_0,r), r>0$ and $\eta >0$ and a holomorphic function $F:D(z_0,r)\rightarrow 
\mathbb{C}$, such that $\gamma_{i}((t_0-\eta,t_0+\eta)\cap I_i) \subset
D(z_0,r)$ and $ F(\gamma_{i}(t))=f(\gamma_{i}(t)) $ for all 
$t\in(t_0-\eta,t_0+\eta)\cap I_i$. This holds true because of the following observations:\\
\begin{enumerate}
\item For some constant $\eta >0$ we can find $ r>0 $ such that $ 
D(z_0,r)\cap \gamma_{i}^{*}\subseteq \gamma_{i}((t_0-\eta,t_0+\eta)\cap I_i)$.
This follows from the fact that the disjoint compact sets $\gamma_i[I_i 
\setminus(t_0-\eta,t_0+\eta)]$ and $\{z_0\}$ have a strictly positive distance. \\
\item For some constant $ r>0 $ we can find $\eta >0$ such that $\gamma_{i}
((t_0-\eta,t_0+\eta)\cap I_i) \subset D(z_0,r)$, because of the continuity of the map 
$\gamma_i$.
\end{enumerate}
\end{remark}

\begin{theorem}
Let $p\in \{0,1,\dots\}\cup \{\infty\}$ and $\Omega$ be a bounded domain in $\mathbb{C}$ 
defined by a finite number of disjoint Jordan curves. Let also $z_0 \in \partial \Omega$. The class of non-holomorphically extendable at $z_0$ functions of $A^p(\Omega)$ is a dense and $G_\delta$ subset of $A^p(\Omega)$.
\end{theorem}

\begin{proof}
Let $M>0$, $r>0$ and $A(p,\Omega,z_0,r,M)$ be the class of functions $f\in A^p(\Omega)$ for 
which there exist a holomorphic function $F: D(z_0,r) \rightarrow \mathbb{C}$ such that
$F|_{D(z_0,r)\cap\partial\Omega}=f|_{D(z_0,r)\cap\partial\Omega}$ and $|F(z)|\leq M$ for
$z\in D(z_0,r)$. We will show that this class is a closed subset of $A^p(\Omega)$ with
empty interior.

Similarly to the proof of Lemma $4.1$ the class $A(p,\Omega,z_0,r,M)$ is a closed subset of 
$A^p(\Omega)$.

If $A(p,\Omega,z_0,r,M)$ does not have empty interior, then there exist a function $f$ in the 
interior of $A(p,\Omega,z_0,r,M)$, a number $b\in \{0,1,2,\dots\}$ and $ \delta>0 $ such that 
\begin{align*}
\{g \in A^p(\Omega ): \sup \limits_{z \in \overline{\Omega} }\left| f^{(j)}(z)-g^{(j)}(z) 
\right| < \delta,\\ 0\leq j \leq b \}\subset A(p,\Omega,z_0,r,M).
\end{align*}
We choose $w\in D(z_0,r)\setminus \overline{\Omega}$ and $$ 0<a<\delta \min\{
\inf \limits_{z\in \overline{\Omega}}|z-w|,\inf \limits_{z\in \overline{\Omega}}|z-w|^2,\dots, 
\dfrac{1}{b!}\inf \limits_{z\in \overline{\Omega}}|z-w|^{b+1} \}.$$
This is possible because $w\not\in \overline{\Omega}$. Then similarly to the proof of 
Theorem $4.7$ we are led to a contradiction. Thus $A(p,\Omega,z_0,r,M)$ has empty interior.

Notice that the set $$\bigcap\limits_{M=1}^{\infty} \bigcap\limits_{n=1}^{\infty}\left(A^p(\Omega)
\setminus A\left(p,\Omega,z_0,\dfrac{1}{n},M\right)\right)$$ coincides with the class of 
non-holomorphically 
extendable at $z_0$ functions of $A^p(\Omega)$ and Baire's theorem implies that this set
is a dense and $G_\delta$ subset of $A^p(\Omega)$.
\end{proof}

\begin{defi}
Let $\Omega$ be a bounded domain in $\mathbb{C}$ defined by a finite number of
disjoint Jordan curves. A continuous function $f: \overline{\Omega}\rightarrow \mathbb{C}$
belongs to the class of nowhere holomorphically extendable functions if for every 
$z_0\in\partial \Omega$, $f$
belongs to the class of non-holomorphically extendable at $z_0$ functions.
\end{defi}

\begin{theorem}
Let $p\in \{0,1,\dots\}\cup \{\infty\}$ and let $\Omega$ be a bounded domain in $\mathbb{C}$ 
defined by a finite number of
disjoint Jordan curves. The class of nowhere holomorphically 
extendable functions of $A^p(\Omega)$ is a dense and $G_\delta$ subset of $A^p(\Omega)$.
\end{theorem}

\begin{proof}
The proof is similar to the proof of Theorem $4.9$, taking into account the statement of 
Theorem $5.18$.
\end{proof}

\begin{remark}
If the continuous analytic capacity of the boundary of $\Omega$ is zero,
then Definition $5.11$ implies Definition $5.16$.

Indeed, let  $\Omega$ be a bounded domain 
in $\mathbb{C}$ defined by a 
finite number of disjoint Jordan curves, such that the continuous analytic capacity of
$\partial\Omega$ is zero, and let $z_0\in\partial \Omega$. Let $f$ be a continuous function in 
$\overline{\Omega}$ which does 
not belong to the class of Definition $5.16$; That is there exist a
pair of an open disk $D(z_0,r)$, $r>0$ and a holomorphic function 
$F :D(z_0,r) \rightarrow \mathbb{C}$ such that $F(z)=f(z)$ for every 
$z\in D(z_0,r)\cap\partial\Omega$. We consider the function $G: D(z_0,r) \rightarrow 
\mathbb{C}$ 
$$ G(z) = \left\{
\begin{array}{ c l }
F(z),   &    z\in D(z_0,r) \setminus \Omega  \\
f(z),   &    z\in D(z_0,r) \cap  \Omega
\end{array}
\right. $$
Then $G$ is continuous on $D(z_0,r)$ and holomorphic on $D(z_0,r)\setminus \partial\Omega$. 
But the continuous analytic capacity of $\partial\Omega$ is zero.
From Theorem $3.6$ $G$ is holomorphic on $D(z_0,r)$ and since the set $D(z_0,r) \cap  \Omega$
is an open subset of $\mathbb{C}$ there exist $z_1\in D(z_0,r) \cap \Omega$ and $d>0$ such 
that $D(z_1,d)\subset D(z_0,r) \cap \Omega$. Obviously, $G$ coincides with $f$ on $D(z_1,d)$
and thus $f$ does not belong to the class of Definition $5.11$.
\end{remark}

Now as in section $4$, we will associate the phenomenon of non-extendability with that of
real analyticity on the spaces $A^p(\Omega)$.

\begin{defi}
Let  $n\in \{1,2,\dots\}$ and let $\Omega$ be a bounded domain in $\mathbb{C}$ defined
by disjoint Jordan curves $\gamma_1,\dots,\gamma_n$.
A function $f: \overline{\Omega}\rightarrow \mathbb{C}$ is real analytic at $(t_0,
\gamma_i(t_0))$, $\gamma_i(t_0)\in \gamma_i^*$, 
$i\in\{1,2,\dots,n\}$ if $f|_{\gamma_i}$ is real analytic at $(t_0,\gamma_i(t_0))$.
\end{defi}

At this point we observe that if $\Omega$ is a bounded domain in $\mathbb{C}$ defined
by disjoint Jordan curves $\gamma_1,\dots,\gamma_n$ and $f\in A^p(\Omega)$, then the 
analogous of Proposition $2.9$ under the above assumptions holds true, since nothing 
essential changes in its proof. So, we have the following proposition:

\begin{prop}
Let $p\in \{0,1,\dots\} \cup \{\infty\}$, $n\in \{0,1,\dots\}$ and let $\Omega$ be a bounded 
domain in $\mathbb{C}$ defined 
by disjoint analytic Jordan curves $\gamma_1,\dots,\gamma_n$. A continuous function $f: \overline{\Omega}
\rightarrow \mathbb{C}$ is real analytic at $(t_0,\gamma_i(t_0))$, 
$\gamma_i(t_0)\in \gamma_i^*$, $i\in\{1,2,\dots,n\}$ if and only if is holomorphically 
extendable at $\gamma_i(t_0)$.
\end{prop}

\begin{defi}
Let $p\in \{0,1,\dots\} \cup \{\infty\}$, $n\in \{0,1,\dots\}$ and let $\Omega$ be a bounded 
domain in $\mathbb{C}$ defined by disjoint Jordan curves $\gamma_1,\dots,\gamma_n$.
A function $f\in A^p(\Omega)$ is nowhere real analytic if there exist no $i\in\{1,2,\dots,n\}$
and $\gamma_i(t_0)\in \partial\Omega$ such that $f$ is real analytic at
$(t_0,\gamma_i(t_0))$.
\end{defi}

Now combining Proposition $5.23$ with Theorem $5.20$ we obtain the following theorem.

\begin{theorem}
Let $p\in \{0,1,\dots\} \cup \{\infty\}$ and let $\Omega$ be a bounded domain in $\mathbb{C}$ 
defined
by a finite number of disjoint analytic Jordan curves. The class of nowhere real analytic 
functions of $A^p(\Omega)$ is a dense and $G_\delta$ subset of $A^p(\Omega)$.
\end{theorem}

\begin{remark}
We recall that for an analytic Jordan curve $\gamma$ defined on $[0,1]$ there exist 
$0<r<1<R$ and a holomorphic injective function $\Phi :D(0,r,R)\rightarrow\mathbb{C}$,
such that $\gamma(t)=\Phi(e^{it})$, where $D(0,r,R)=\{z\in\mathbb{C}: r<|z|<R\}$. 
This yields a natural parametrization of the curve 
$\gamma^*$; the parameter $t$ is called a conformal parameter for the curve $\gamma^*$.
Theorem $5.25$ holds if each of the Jordan curves $\gamma_1,\dots,\gamma_n$ is parametrized by 
such a conformal parameter t. Naturally one asks if the same result holds for other 
parametrizations; for instance does Theorem $5.25$ remains true if each 
$\gamma_1,\dots,\gamma_n$ is parametrized by arc length? This was the motivation of \cite{[9]} 
and \cite{[10]}, where 
it is proved that arc length is a global conformal parameter for any analytic curve. Thus
Theorem $5.25$ remains also true if arc length is used as a parametrization for each analytic 
curve $\gamma_i$.
\end{remark}

\section{One sided extendability}

In this section we consider one sided extensions from a locally injective curve $ \gamma $. 
For instance, if $\gamma^*$ is homeomorphic  to $[0,1]$, one can find an open disc $D$ and 
an open arc $J$ of $ \gamma^{*} $ separating $ D $ to two components $D^{+}$ and $D^{-}$. 
Those are Jordan domains containing in their boundaries a subarc $J$ of  $ \gamma $. We will 
show that generically in $C^{k}(\gamma)$ every function $h$ cannot be extended to a function 
$F:\Omega\cup J\rightarrow \mathbb{C}$ continuous on $\Omega\cup J$ and holomorphic on $ 
\Omega $; where $\Omega=D^{+} $  or $\Omega=D^{-} $. That is the one sided extendability is a 
rare phenomenon in $C^{k}(\gamma)$, provided that $\gamma $ is of class at least $C^k$. In 
order to prove this fact we need the following lemmas, which are well known in algebraic
topology. We include their elementary proofs for the purpose of completeness.

\begin{lemma}
Let $\delta:[0,1]\rightarrow\mathbb{C}$ be a 
continuous and injective curve. Then the interior of $\delta^*$ in $\mathbb{C}$ is void.
\end{lemma}

\begin{proof}
We will prove the lemma by contradiction. Suppose that there exists an open disk 
$W\subset\delta^*$. Let $t_s$ be a point in the interior of $[0,1]$, such that 
$\delta(t_s)\in W$. Since the function
$\delta$ is a homeomorphism of $[0,1]$ on $\delta^*\subset \mathbb{C}$, then $\delta|_{[0,1] 
\setminus \{ t_s \}}$ is a homemorphism of $[0,1] \setminus \{ t_s \}$ on 
$\delta^*\setminus\{ \delta(t_s)\}\subset \mathbb{C}$. But the set
$[0,1] \setminus \{ t_s \}$ has two connected components and thus the set
$\delta^*\setminus\{ \delta(t_s)\}$ has two connected components $V_1,V_2$. 
Both $V_1,V_2$ intersect $W$, because
we can find points of both $V_1,V_2$ arbitrary close to the point $\delta(t_s)$ of the open
set $W$. Therefore $W\cap V_1 \neq \emptyset$, $W\cap V_2 \neq \emptyset$ and $W\setminus
\{\delta(t_s)\}=(W\cap V_1) \cup(W\cap V_2)$. Also, the sets $V_1, V_2$ are closed in the
relative topology of $\delta^*\setminus\{ \delta(t_s)\}\supset W\setminus\{ \delta(t_s)\}$.
It follows that the sets $W\cap V_1 , W\cap V_2$ are closed in the relative topology of
$W\setminus\{\delta(t_s)\}$. Consequently, the set $W\setminus\{\delta(t_s)\}$ is not 
connected, which is absurd, since it is an open disk without one of its interior points.
Hence the interior of $\delta^*$ in $\mathbb{C}$ is void.
\end{proof}

\begin{prop}
Let $\gamma:I\rightarrow\mathbb{C}$ be a continuous and locally injective curve. Then the interior of $\gamma^*$ in $\mathbb{C}$ is void.
\end{prop}

\begin{proof}
Let $t_i\in I$ be such that $\gamma|_{[t_i,t_{i+1}]}$ is injective, where $i$ varies on a 
finite or infinite denumerable set $ S\subset \mathbb{Z} $ and
$$\bigcup_{n\in S}[t_n,t_{n+1}]=I.$$
From Lemma 6.1, the interior of every $\gamma_{[t_i,t_{i+1}]}^*$ is 
void in $\mathbb{C}$. Also, the set $$\bigcup_{n\in S}\gamma_{[t_n,t_{n+1}]}^*$$
coincides with $\gamma^*$ and Baire's theorem implies that the interior of $\gamma^*$ is 
void in $\mathbb{C}$.
\end{proof}

Proposition $6.2$ implies the following.

\begin{corol}
Let $\gamma:I\rightarrow\mathbb{C}$ be a continuous and locally injective curve and $\Omega$ be a Jordan domain whose boundary contains an arc $\gamma([t_1,t_2])$, $t_1<t_2$, $t_1,t_2 \in I$ of 
$\gamma^*$. Then the set $\Omega\setminus\gamma^*$ is non-empty.
\end{corol}

Let $\gamma:I\rightarrow\mathbb{C}$ be a continuous and locally injective curve defined
on the interval $I\subset\mathbb{R}$. Naturally one asks if a Jordan domain as in 
Corollary $6.3$ exists. Our goal is to construct denumerably many such Jordan domains, so 
that every $ \Omega $, as in Corollary $6.3$, contains one of these domains and then use Baire's Category Theorem. 

Let $t_0$ in the interior $I^\circ$ of $I$ and $\Omega$ be a Jordan domain whose boundary contains an arc $\gamma([t_1,t_2])$, $t_1<t_0<t_2$. To ease notation let us assume that $t_1,t_2$ are rational numbers. Pick $s_1,s_2 \in (t_1,t_2)\cap \mathbb{Q}$ such that $s_1<t_0<s_2$. By definition the point $\gamma(s_1)$ does not belong to the compact set $\Omega_1:=\partial\Omega\setminus\gamma(t_1,t_0)$. As a result their distance $\delta:=dist(\gamma(s_1),\Omega_1)$ is positive. From Proposition $6.2$ there exists
$P\in (\mathbb{Q}+i\mathbb{Q})\cap D(\gamma(s_1),\delta/2)\cap \Omega$ with
$P\not\in\gamma([t_1,t_0])$. Notice that there is $r\in (t_1,t_0)$ such that $|\gamma(r)-P|=a$, where $a:=dist(P,\gamma([t_1,t_0]))\leq |P-\gamma(s_1)|<\delta/2$, because for every $z\in \Omega_1$ we have $$|z-P|\geq|z-\gamma(s_1)|-|\gamma(s_1)-P|> \delta-\delta/2=\delta/2.$$ By definition the segment $[P,\gamma(r)]$ intersects $\partial \Omega$ only at $\gamma(r)$. Similarly we pick $Q\in (\mathbb{Q}+i\mathbb{Q})\cap \Omega$ and $\widetilde{r}\in (t_0,t_2)$, such that the segment $[Q,\gamma(\widetilde{r})]$ intersects $\partial \Omega$ only at $\gamma(\widetilde{r})$. 

We distinguish two cases according to whether the segments $[P,\gamma(r)]$,
$[Q,\gamma(\widetilde{r})]$ intersect or not. If the segments $[P,\gamma(r)]$,
$[Q,\gamma(\widetilde{r})]$ intersect at a point $w$, then the
union of the segments $[w,\gamma(r)]$, $[w,\gamma(\widetilde{r})]$ and 
$\gamma[\widetilde{r},r]$ is the image of a Jordan curve, 
the interior of which is the desired Jordan domain.
If the segments $[P,\gamma(r)]$,
$[Q,\gamma(\widetilde{r})]$ do not intersect, then we consider a simple polygonal line 
(that is without self intersections) in
$\Omega$, which connects $P,Q$, the vertices of which belong to 
$\mathbb{Q}+i\mathbb{Q}$.
This is possible, since $\Omega$ is a domain (\cite{[11]}).
The union of the aforementioned polygonal line with the the segments $[P,\gamma(r)]$ and
$[Q,\gamma(\widetilde{r})]$ has possibly self-intersections. For this reason we consider one of the connected components of this simple polygonal line minus
the segments $[P,\gamma(r)]$,
$[Q,\gamma(\widetilde{r})]$ with the property that the closure of this connected component is a simple polygonal line connecting two points $z_1 \in [P,\gamma(r)]$, 
$z_2 \in [Q,\gamma(\widetilde{r})]$. Then the union of the last simple polygonal line
with the segments $[z_1,\gamma(r)]$, 
$[z_2,\gamma(\widetilde{r})]$ and the arc $\gamma[\widetilde{r},r]$ of 
$\gamma^*$ is the image of a Jordan curve, 
the interior of which is the desired Jordan domains. We notice that
$t_0\in [\widetilde{r},r]$ and that the constructed Jordan domains are denumerably many.

\begin{prop}
Let $\gamma:I\rightarrow\mathbb{C}$ be a locally injective map of class $C^l(I)$, $l\in
\{0,1,2,\dots\}\cup\{\infty\}$ and $0<M<\infty$. 
Let also $\Omega$ be a Jordan domain whose boundary contains an arc $\gamma([t_1,t_2])$, $t_1<t_2$, $t_1,t_2 \in I$ of 
$\gamma^*$ and $k\in\{0,1,2,\dots\}\cup\{\infty\}$, $k\leq l$. The set of functions 
$f\in C^k(\gamma)$ for which there exists a continuous function $F:\Omega\cup
\gamma((t_1,t_2))$, $\lVert F \rVert_{\infty}\leq M$, 
such that $F$ is holomorphic on $\Omega$ and $F|_{\gamma((t_1,t_2))}=f|_{\gamma((t_1,t_2))}$, 
is a closed subset of $C^k(\gamma)$ with empty interior.
\end{prop}

\begin{proof}
Let $\Psi$ be a homeomorphism of $D\cup J\subset\mathbb{C}$
on $\Omega\cup \gamma(t_1,t_2)$, which is also holomorphic on 
$\Omega$, where $J=\{e^{it}:0<t<1 \}$. This is possible because of the Caratheodory-Osgood
theorem. Let also $A(k,\Omega,M)$ be the set of functions $ f\in C^k(\gamma) $ for which 
there exists a continuous function $F:\Omega\cup\gamma((t_1,t_2))\rightarrow \mathbb{C}$, $\lVert F \rVert_{\infty}\leq M$, 
such that $F$ is holomorphic on $\Omega$ and $F|_{\gamma((t_1,t_2))}=f|_{\gamma((t_1,t_2))}$.

First, we will prove that $A(k,\Omega,M)$ is a closed subset of $C^k(\gamma) $. 
Let $ (f_n)_{n\geq 1} $ be a sequence in $A(k,\Omega,M)$
converging in the topology of $C^k(\gamma) $ to a function $ f\in C^l(\gamma)$. This implies 
that $f_n$ converges uniformly on the compact subsets of $\gamma^*$ to $f$. Then for 
$ n=1,2, \dots $ there exist continuous functions $F_n: \Omega\cup\gamma((t_1,t_2))\rightarrow \mathbb{C}$, 
$\lVert F_n \rVert_{\infty}\leq M$, 
such that $F_n$ are holomorphic on $\Omega$ and 
$F_n|{ _{\gamma((t_1,t_2))}}=f_n|{_{\gamma((t_1,t_2))}}$. If $ G_n=F_n\circ\Psi $,
$ g_n=f_n\circ\Psi $ and $ g=f\circ\Psi $ for $ n=0,1,2, \dots $, it follows that $ g_n $ 
converges uniformly on the compact subsets of $J$ 
to $ g $. Also the functions $ G_n $ are 
holomorphic and bounded by $ M $ on $D$. By Montel's theorem, there exists a subsequence of 
$ (G_n) $, $ (G_{k_n}) $ which converges uniformly on the compact subsets of $ D $ to a 
function $ G $ which is holomorphic and bounded by $ M $ on $ D $. Without 
loss of generality, we assume that $ (G_n)=(G_{k_n}) $.
Now it is sufficient to prove that for any circular sector $K$, which has boundary 
$ [0,e^{ia}]\cup [0,e^{ib}] \cup \left\{ e^{it}: a\leq t\leq b \right\}$, $ 0<a<b<1 $, 
the sequence $ (G_n) $ converges uniformly on $K$, because then the limit of 
$ (G_n) $, which is equal to $ g $ at the arc $J$ and equal to $ G $ on the remaining part 
of the circular sector $K$, will be a 
continuous function. In order to do so we will prove that $(G_n)$ forms a
uniformly Cauchy sequence on $ K $. 
Since each $G_n$ is a bounded holomorphic function on $D$, we know that for every $ n $ 
the radial limits of $ G_n $ exist almost 
everywhere on the unit circle and so we can consider the respective functions $ g_n $ defined 
almost everywhere on the unit circle which 
are extensions of the previous $ g_n $. These $ g_n $ are also bounded by $ M $.\\
Let $ \varepsilon>0 $ be a positive number. For the Poisson kernel $P_r$, $0\leq r <1$ and 
for every 
$ n=0,1,2 \dots $, it holds that $G_n(re^{i\theta})=\dfrac{1}{2\pi}\int\limits_{-\pi}^{\pi}
P_r(t)g_n(\theta-t)dt $. We choose $ 0<\delta < min\{1-b,a\} $. There exists $ 0<r_0<1 $ 
such that $ \sup\limits_{\delta\leq |t|\leq \pi} P_r(t) < \frac{\varepsilon}{8M} $ for every 
$ r\in [r_0,1) $. Then $(G_n)$ is a
uniformly Cauchy sequence on $ K\cap \{z \in \mathbb {C}: |z|\leq r_0\}$, and thus there 
exists $ n_1 $ such that for every $ n,m\geq n_1 $, 
$$ \sup\limits_{z \in K\cap \{z \in \mathbb {C}: |z|\leq r_0\} }|G_n(z)-G_m(z)|<
\dfrac{\varepsilon}{2}   \qquad   (1)$$ 
In addition, as $ g_{n} $ converges uniformly to $ g $ on $ J $ there exists $ n_2 $, 
such that for every 
$ n,m\geq n_2 $, $ \sup\limits_{z \in J} |g_n(z)-g_m(z)|<\dfrac{\varepsilon}{4}$.
Consequently, for $ n,m\geq \max\{n_1,n_2\} $, for $ \theta \in [a,b] $ and for $ 1>r>r_0 $
$$ |G_n(re^{i\theta})-G_m(re^{i\theta})|\leq  \dfrac{1}{2\pi}\int\limits_{-\pi}^{\pi}P_r(t)|
g_n(\theta-t)-g_m(\theta-t)|dt=$$
$$\dfrac{1}{2\pi}\int\limits_{-\delta}^{\delta}P_r(t)|g_n(\theta-t)-g_m(\theta-t)|dt+$$
$$+\dfrac{1}{2\pi}\int\limits_{\delta\leq |t|\leq \pi}P_r(t)|g_n(\theta-t)-g_m(\theta-t)|dt
\leq 
$$
$$\dfrac{1}{2\pi}\int\limits_{-\delta}^{\delta}P_r(t)\sup\limits_{z \in J} |g_n(z)-g_m(z)|dt
+\dfrac{\sup\limits_{\delta\leq |t|\leq \pi} P_r(t)}{2\pi}\int\limits_{\delta\leq |t|\leq 
\pi}2Mdt\leq $$
$$\dfrac{\varepsilon}{4}\dfrac{1}{2\pi}\int\limits_{-\delta}^{\delta}P_r(t)dt+
\dfrac{\varepsilon}{16M\pi}\int\limits_{\delta\leq |t|\leq \pi}2Mdt\leq $$
$$\dfrac{\varepsilon}{4}\dfrac{1}{2\pi}\int\limits_{-\pi}^{\pi}P_r(t)dt+\dfrac{\varepsilon}
{4}=\dfrac{\varepsilon}{2} \qquad (2) .$$
By the continuity of the functions $G_n$ on $D\cup J$, 
making $r\rightarrow 1^-$ we find that
$$ |G_n(e^{i\theta})-G_m(e^{i\theta})| \leq \dfrac{\varepsilon}{2}, \qquad (3)$$ for all 
$\theta\in[a,b]$. It follows immediately from $ (1), (2) $ and $(3)$ that $ (G_n) $ is a uniformly Cauchy sequence on the circular sector
$K$ and thus the set $A(k,\Omega,M)$ is a closed subset of $ C^k(\gamma) $.

If $A(k,\Omega,M)$ does not have empty interior, then there exists a function $ f $ 
in the interior of $A(k,\Omega,M)$, a compact set $L\subset I $ and $ \delta>0 $ 
such that 
\begin{align*}
\{g \in C^k(\gamma ): \sup \limits_{t \in L}\left| (f\circ \gamma)^{(j)}(t)-(g\circ 
\gamma)^{(j)}(t) \right| < 
\delta,\\ 0\leq j \leq b \}\subset A(k,\Omega,M).
\end{align*}
From Corollary $6.3$, we can find a $w\in \Omega\setminus \gamma^*$. 
We choose $ 0<a<\delta \min\{
\inf \limits_{t\in K}|\gamma(t)-w|,\inf \limits_{t\in K}|\gamma(t)-w|^2,\dots, 
\dfrac{1}{b!}\inf \limits_{t\in K}|\gamma(t)-w|^{b+1} \}.$
This is possible because $w\not\in\gamma^*$ and $\gamma(K)\subset \gamma^*$. The function 
$ h(\gamma(t))=f(\gamma(t))+ \dfrac{a}{2(\gamma(t)-w)} $, $ t\in 
I$ belongs to $A(k,\Omega,M)$ and therefore has a continuous and bounded extension $ H $ on
$\Omega\cup\gamma((t_1,t_2))$ with $ H|
_{\gamma((t_1,t_2))}=h|_{\gamma((t_1,t_2))} $ which is holomorphic on $\Omega$. 
Then the function $H\circ \Psi$ is continuous and bounded on $D\cup J$ and holomorhic
on $D$. We can easily see 
that $ H(\Psi(z))=F(\Psi(z))+ \dfrac{a}{2(\Psi(z)-w)} $ for 
$z\in D \setminus \{\Psi^{-1}(w)\}$. Indeed, let $ \Phi(z)=H(\Psi(z))-F(\Psi(z))- \dfrac{a}{2(\Psi(z)-w)} $. Then $ \Phi|_J=0 $ and by Schwarz Reflection 
Principle $\Phi$ is extended holomorphically on 
$$(D\cup J\cup(\mathbb{C}\setminus\overline{D}))\setminus \{\Psi^{-1}(w),\dfrac{1}
{\overline{\Psi^{-1}(w)}}\}.$$  Therefore because $ \Phi=0 $ on $J$, by analytic 
continuation, $ H(\Psi(z))-F(\Psi(z))- \dfrac{a}{2(\Psi(z)-w)}=0$ on 
$(D\cup J)\setminus \{\Psi^{-1}(w)\}$. As a result $ H\circ \Psi $ 
is not bounded on $D$ which is absurd. Thus $A(k,\Omega,M)$ has empty interior.
\end{proof}

\begin{defi}
Let $\gamma:I\rightarrow\mathbb{C}$ be a locally injective map and $t_0\in I^\circ$, where $I^\circ$ is the interior of $I$ in $\mathbb{R}$. A function $f:\gamma^*\rightarrow\mathbb{C}$ is non one-sided holomorphically extendable at $(t_0,\gamma(t_0))$ if there exists
no pair of a Jordan domain $\Omega$, such that 
$\partial \Omega$ contains an arc of 
$\gamma^*$, $\gamma([t_1,t_2])$, $t_1<t_0<t_2$, $t_1,t_2 \in I$ and a continuous function 
$F:\Omega\cup\gamma((t_1,t_2))$, 
which is holomorphic on $\Omega$ and $F|_{\gamma((t_1,t_2))}=f|_{\gamma((t_1,t_2))}$.
\end{defi}

\begin{theorem}
Let $\gamma:I\rightarrow\mathbb{C}$ be a locally injective map of class $C^l(I)$, $l\in
\{0,1,2,\dots\}\cup\{\infty\}$ and $k\in\{0,1,2,\dots\}\cup\{\infty\}$, $k\leq l$. 
Consider $t_0\in I^\circ$, where $I^\circ$ is
the interior of $I$ in $\mathbb{R}$. The class of non
one sided holomorphically extendable at $(t_0,\gamma(t_0))$ functions of $C^k(\gamma)$
is a dense and $G_\delta$ subset of $C^k(\gamma)$.
\end{theorem}

\begin{proof}
Let $\Omega$ be a Jordan domain, such that $\partial \Omega$ contains an arc $\gamma([t_1,t_2])$,  of $\gamma^*$, where $t_1<t_0<t_2$. $A(k,\Omega,M)$ denotes the set of functions 
$f\in C^k(\gamma)$ for which there exists a continuous function 
$F:\Omega\cup\gamma((t_1,t_2))$, which is holomorphic on $\Omega$, bounded by $M$ and $F|_{\gamma((t_1,t_2))}=f|_{\gamma((t_1,t_2))}$.
Let $G_n$, $n=1,2,\dots$ be the denumerably many Jordan domains constructed above.

From Proposition $6.4$, the sets $A(k,G_n, M)$ are closed subsets of
$C^k(\gamma)$ with empty interior. We will prove that the class 
of non one sided holomorphically extendable at $(t_0,\gamma(t_0))$ functions of $C^k(\gamma)$  
coincides with the set
$$\bigcap_{n=1}^\infty \bigcap_{M=1}^\infty(C^k(\gamma)\setminus A(k,G_n,M)),$$ 
and thus Baire's theorem
will imply that the above set is a dense and $G_\delta$ subset of $C^k(\gamma)$.

Obviously, the set 
$$\bigcap_{n=1}^\infty\bigcap_{M=1}^\infty(C^k(\gamma)\setminus A(k,G_n,M))$$ 
contains the class of non one-sided holomorphically extendable at $(t_0,\gamma(t_0))$ functions of $C^k(\gamma)$. Conversely, let $\Omega$ be a Jordan domain whose boundary contains an arc
$\gamma([t_1,t_2])$, $t_1<t_0<t_2$, $t_1,t_2\in I$. Let also $f\in C^k(\gamma)$, 
for which there exists a continuous function 
$F:\Omega\cup\gamma((t_1,t_2))$, 
which is holomorphic on $\Omega$ and $F|_{\gamma((t_1,t_2))}=f|_{\gamma((t_1,t_2))}$.
From the construction of the aforementioned Jordan domains $G_n$, we can find a Jordan domain $G_{n_0}$ such that $\overline{G_{n_0}}$ is contained in $\overline{\Omega}$.
It easily follows that $F|_{G_{n_0}}$ is bounded by some number $M=1,2,3,\dots$ and thus $f$ 
belongs to $A(k,G_{n_0},M)$. Therefore the class of non
one sided holomorphically extendable at $(t_0,\gamma(t_0))$ functions of $C^k(\gamma)$ 
is a subset of the set 
$$\bigcap_{n=1}^\infty\bigcap_{M=1}^\infty(C^k(\gamma)\setminus A(k,G_n,M)),$$
which combined with the above completes the proof.
\end{proof}

\begin{defi}
Let $\gamma:I\rightarrow\mathbb{C}$ be a locally injective map on the interval 
$I\subset\mathbb{R}$. A function $f:\gamma^*\rightarrow\mathbb{C}$ is nowhere
one-sided holomorphically extendable if there exists
no pair of a Jordan domain $\Omega$, such that 
$\partial \Omega$ contains an arc of 
$\gamma^*$, $\gamma([t_1,t_2])$, $t_1<t_2$, $t_1,t_2 \in I$ and a continuous function 
$F:\Omega\cup\gamma((t_1,t_2))$, 
which is holomorphic on $\Omega$ and $F|_{\gamma((t_1,t_2))}=f|_{\gamma((t_1,t_2))}$.
\end{defi}

\begin{theorem}
Let $\gamma:I\rightarrow\mathbb{C}$ be a locally injective map of class $C^l$, $l\in
\{0,1,2,\dots\}\cup\{\infty\}$ on the interval $I\subset\mathbb{R}$ and $k\in
\{0,1,2,\dots\}\cup\{\infty\}$, $k\leq l$. 
The class of nowhere one-sided holomorphically extendable functions of $C^k(\gamma)$
is a dense and $G_\delta$ subset of $C^k(\gamma)$.
\end{theorem}

\begin{proof}
Let $t_n\in I^\circ$, $n=1,2,\dots$. Then the class of nowhere one sided holomorphically extendable functions of $C^k(\gamma)$ coincides with the intersection of the classes of non 
one sided holomorphically extendable at $(t_n,\gamma(t_n))$ functions of $C^k(\gamma)$,
which is from Theorem $6.6$ and Baire's theorem a dense and $G_\delta$ subset of $C^k(\gamma)$.
\end{proof}

Let $n \in\{1,2,\dots\}$ and  $\gamma_i : I_i \rightarrow \mathbb{C}$, $i=1,\dots,n$, be 
continuous and locally injective curves. We recall the definition of the space
$C^{p_1,\dots,p_n}(\gamma_1,\dots,\gamma_n)=C^{p_1}(\gamma_1)\times\dots\times C^{p_n}(\gamma_n)$,
where $p_i\in\{0,1,2,\dots\}\cup\{\infty\}$ for $i=1,\dots,n$. As we did in section $5$,
we will prove generic results for the spaces $C^{p_1,\dots,p_n}(\gamma_1,\dots,\gamma_n)$.

\begin{defi}
Let $n \in\{1,2,\dots\}$ and $\gamma_i : I_i \rightarrow \mathbb{C}$, $i=1,\dots,n$, be 
locally injective curves, where $I_i$ are intervals and $t_0$ a point of some $I_{i_0}$. A function $f$ on the disjoint union $\gamma_1^*\cup\dots\cup\gamma_n^* $, $f(z)=f_i(z)$ for $z\in
\gamma_i^*$, is non oned sided holomorphically extendable at $(t_0,\gamma_{i_0}(t_0))$ if 
the function $f_{i_0}$ defined on $ \gamma^*_{i_0} $ is 
non one sided holomorphically extendable at $(t_0,\gamma_{i_0}(t_0))$.
\end{defi}

\begin{theorem}
Let $n \in\{1,2,\dots\}$ and $p_i,q_i \in\{0,1,\dots\}\cup \{\infty\}$ such that $p_i\leq q_i$ for every $i=1,\dots,n$. Let also $\gamma_i:I_i \rightarrow \mathbb{C}$ be locally 
injective functions of $C^{q_i}(I_i)$, where $I_i$ are intervals, and $t_0$ a point of some $I_{i_0}$. The class of non one sided homolomorphically extendable at $(t_0,\gamma_{i_0}(t_0))$ functions belonging to $C^{p_1,\dots,p_n}(\gamma_1,\dots,\gamma_n)$ is a dense and $G_\delta$ subset of $C^{p_1,\dots,p_n}(\gamma_1,\dots,\gamma_n)$.
\end{theorem}

\begin{proof}
Similar to the proof of Theorem $5.5$.
\end{proof}

\begin{defi}
Let $n \in\{1,2,\dots\}$ and $\gamma_i : I_i \rightarrow \mathbb{C}$ be 
locally injective curves, where $I_i$ are intervals. A function 
$f$ on the disjoint union $\gamma_1^*\cup\dots\cup\gamma_n^* $, $f(z)=f_i(z)$ for $z\in
\gamma_i^*$, is nowhere one sided holomorphically extendable if the functions $f_i$ defined on 
$ \gamma^*_i $ are nowhere one sided holomorphically extendable functions.
\end{defi}

The proof of the following theorem is also similar to the proof of Theorem $5.5$.

\begin{theorem}
Let $n \in\{1,2,\dots\}$, $p_i,q_i \in\{0,1,\dots\}\cup \{\infty\}$ such that $p_i\leq q_i$ for 
$i=1,\dots,n$. Let also $\gamma_i:I_i \rightarrow \mathbb{C}$ be locally 
injective functions of $C^{q_i}(I_i)$, where $I_i$ are intervals. 
The class of nowhere one sided homolomorphically extendable functions belonging 
to $C^{p_1,\dots,p_n}(\gamma_1,\dots,\gamma_n)$ is a dense and $G_\delta$ subset of 
$C^{p_1,\dots,p_n}(\gamma_1,\dots,\gamma_n)$.
\end{theorem}

\begin{defi}
Let $n \in\{1,2,\dots\}$ and $\Omega$ be a bounded domain in $\mathbb{C}$ defined by 
disjoint Jordan curves $\gamma_i:\mathbb{R} \rightarrow \mathbb{C}$, where each $\gamma_i$ is 
continuous and periodic. Let also $t_0\in \mathbb{R}$. A function $f:\overline{\Omega}\rightarrow\mathbb{C}$ is non one sided holomorphically extendable at $(t_0,\gamma_{i_0}(t_0))$ outside $\Omega$ if there exists no pair of a Jordan domain $G\subset\mathbb{C}\setminus\overline{\Omega}$, such that $\partial G$ contains a Jordan arc of 
$\gamma_{i_0}^*$, $\gamma_{i_0}([t_1,t_2])$, $t_1<t_0<t_2$, $t_1,t_2 \in \mathbb{R}$ and a 
continuous function $F:\Omega\cup\gamma((t_1,t_2))$, 
which is holomorphic on $G$ and $F|_{\gamma_{i_0}((t_1,t_2))}=f|_{\gamma_{i_0}((t_1,t_2))}$.
\end{defi}

\begin{theorem}
Let $p\in \{0,1,\dots\}\cup \{\infty\}$, $n \in\{1,2,\dots\}$ and $\Omega$ be a bounded 
domain in $\mathbb{C}$ defined by 
disjoint Jordan curves $\gamma_i:\mathbb{R} \rightarrow \mathbb{C}$, where each $\gamma_i$ is 
continuous and periodic. Let also $t_0\in \mathbb{R}$. The class of non one sided holomorphically 
extendable at $(t_0,\gamma_{i_0}(t_0))$ outside $\Omega$ functions of $A^p(\Omega)$ is a 
dense and $G_\delta$ subset of $A^p(\Omega)$.
\end{theorem}

\begin{proof}
The proof is simply a combination of proofs similar to those of Proposition $6.4$ and Theorem $6.6$.
\end{proof}

\begin{defi}
Let $n \in\{1,2,\dots\}$ and $\Omega$ be a bounded domain in $\mathbb{C}$ defined by 
disjoint Jordan curves $\gamma_i:\mathbb{R} \rightarrow \mathbb{C}$, where each $\gamma_i$ 
is continuous and periodic. A function $f:\overline{\Omega} \rightarrow
\mathbb{C}$ is nowhere one sided holomorphically extendable outside $\Omega$ if $f$ is non
one-sided holomorphically extendable at $(t_0,\gamma_{i_0}(t_0))$  outside $\Omega$ for every 
$t_0\in \mathbb{R}$ and $i_0=1,2,\dots,n$.
\end{defi}

Combining Theorem $6.14$ for a dense in $\mathbb{R}$ sequence $t_n$ 
and Baire's theorem, we obtain the following:

\begin{theorem}
Let $p\in \{0,1,\dots\}\cup \{\infty\}$ and let $\Omega$ be a bounded 
domain in $\mathbb{C}$ defined by a finite number of
disjoint Jordan curves. The class of nowhere one sided holomorphically 
extendable functions outside of $\Omega$ is a dense and $G_\delta$ 
subset of $A^p(\Omega)$.
\end{theorem}

\begin{remark}
In the above statements the functions $ f $ are defined on $ \overline{\Omega} $. However, for 
$ f \in A^p(\Omega) $, by the maximal modulus principle, there is a 
one to one correspondence between the function $ f $ and $ f|_{\partial\Omega} $. Therefore we 
could state the previous result considering $ A^p(\Omega) $ 
restricted to $ \partial\Omega $.
\end{remark}

\section{Removability of singularities in the spaces $A^p$ and the $p$-continuous analytic
capacities. A dichotomy result}

Let $\Omega$ denote an open and bounded subset of $\mathbb{C}$ and $L$ be a compact subset of 
$\Omega$. Consider the open set $G=\Omega\setminus L$, $p\in \{0,1,\dots\}\cup\{\infty\}$ and $f_0\in A^p(G)$. Then there are two cases: 
\begin{enumerate}[(i)]
\item Either there exists a function $F_0\in A^p(\Omega)$, such that $F_0|_G=f_0$
\item or there exists no $F_0\in A^p(\Omega)$, such that $F_0|_G=f_0$.
\end{enumerate}

\begin{theorem}
Let $p\in \{0,1,\dots\}\cup\{\infty\}$, $\Omega$ be an open and bounded subset of 
$\mathbb{C}$ and $L$ be a compact subset of $\Omega$. Let also 
$G=\Omega\setminus L$. Suppose that there is a function $f_0\in A^p(G)$ 
for which there exists no $F_0\in A^p(\Omega)$, such that $F_0|_G=f_0$. Then the set of 
functions $f\in A^p(G)$ for which there exists no $F\in A^p(\Omega)$ such that $F|_G=f$ is an open and dense subset of $A^p(G)$.
\end{theorem}

\begin{proof}
Let $A(p,G)$ be the set of functions $ f\in A^p(G)$ for which 
there exists a continuous function $F\in A^p(\Omega)$ such that $F|_G=f$.

First, we will prove that the set $A(p,G)$ is a closed subset of $ A^p(G)$. Let 
$ (h_n)_{n\geq 1} $ be a sequence in $A(p,G)$ 
converging in the topology of $A^p(G)$ to a function $h\in A^p(G)$. By the maximum modulus principle, the extensions $H_n$ of $h_n$ form a uniformly Cauchy sequence on $\overline{\Omega}$. Thus the limit $H$ of $H_n$ on $\overline{\Omega}$ is an extension of $h$. Therefore $g\in A(p,G)$ and $A(p,G)$ is a closed subset of $A^p(G)$. 

Now we will prove that the set $A(p,G)$ has empty interior in $A^p(G)$. If $A(p,G)$
does not have empty interior in $A^p(G)$, then there is a function $f$ in the interior of $A(p,G)$ and $l\in\{0,1,\dots\}$, $l\leq p$ and $d>0$, such that
$$\{f\in  A^p(G): \sup\limits_{z\in \overline{G}}|f^{(j)}(z)-g^{(j)}(z)|<d, 0\leq j \leq l\}
\subset A(p,G).$$
Obviously, $f_0$ is not identically equal to
zero which implies $$m=max\{\sup\limits_{z\in \overline{G}}|{f_0}^{(j)}(z)|,j=0,1,\dots,l\}>0.$$
From the definitions the function $h(z)=f(z)+\dfrac{d}{2m}f_0(z)$, $z\in \overline{G}$ belongs to 
$A(p,G)$.
Since the functions $f,h$ belong to $A(p,G)$, there are functions $F,H \in A^p(\Omega)$, 
such that $F|_G=f$, $H|_G=h$. Then $\dfrac{2m}{d}(H(z)-F(z))$ belongs to 
$A^p(\Omega)$ and is equal to $f_0$ in $G$ which contradicts our hypothesis. Thus
the set $A(p,G)$ has empty interior in $A^p(G)$.

The set of functions $f\in A^p(G)$ for which there exists no $F\in A^p(\Omega)$, such that $F|_G=f$ coincides with the set $$ A^p(G)\setminus A(p,G)$$ which was proved to be open and dense.
\end{proof}

\begin{remark}
If the interior of $L$ in $\mathbb{C}$ is non-empty, then there always exists a 
function $f_0\in A^p(G)$ for which there does not exist a function $F_0\in A^p(\Omega)$, such that $F_0|_G=f_0|_G$. Indeed let $w\in L^{o}$; then $f_0=\dfrac{1}{z-w}$ belongs to $A^p(G)$ but it can not have an extension in $A^p(\Omega)$. Thus the class 
of functions $f\in A^p(G)$ for which there exists no $F\in A^p(\Omega)$ such that $F|_G=f$ is dense and open in $A^p(G)$.
\end{remark}

\begin{remark}
From the previous results we have a dichotomy: Either every $f\in A^p(G)$ has an extension in 
$A^p(\Omega)$ or generically all functions $f\in A^p(G)$ do not admit any extension 
in $A^p(\Omega)$. The first case holds if and only if $ a_p(L)=0 $ and the second case if and only if $ a_p(L)>0 $ (Theorem $3.12$).
\end{remark}

\begin{remark}
The results of this section can easily be extended to unbounded open sets $ \Omega\subset 
\mathbb{C} $ with the only difference in the definition of the topology of $ A^p(\Omega) $. 
The topology of $ A^p(\Omega) $ is defined by the denumerable family of seminorms 
$$ \sup_{z\in \Omega, |z|\leq n}|f^{(l)}(z)| ,l=0,1,2,\cdots ,n=0,1,2,\cdots .$$
For $ p<\infty $ and $ \Omega $ unbounded $ A^p(\Omega) $ is a Fr{\'e}chet space, while if $ \Omega $ is bounded it is a Banach space.
\end{remark}

\begin{remark}
In a similar way we can prove that if $L$ is a compact set contained in the open set $U$, then either every
$f\in \tilde A^p(U\setminus L) $ has an extension in $\tilde A^p(U)$ or generically every 
$f\in \tilde A^p(U\setminus L) $ does not have an extension in $\tilde A^p(U)$, $p\in\{0,1,2,\cdots\}\cup\{\infty\}$. The first horn of this dichotomy holds if and only if 
$\tilde a_p(L)=0$ which is equivalent with the fact that the interior of $ L $ is void in $\mathbb{C}$ (Theorem $3.16$).
\end{remark}

Now we present some local versions of the results of section $7$.

\begin{defi}
Let $L$ be a compact subset of $ \mathbb{C}$ and $U$ be an open subset of $ \mathbb{C}$, such that $L\subseteq U$. Let also $z_0 \in \partial L$. A function $f \in H(U\setminus L)$ is extendable at $ z_0 $ if there exists $  r>0 $ and $ F\in H(D(z_0, r)) $ such that $ F|_{(U\setminus L)\cap D(z_0,r)}=f|_{(U\setminus L)\cap D(z_0,r)} $. Otherwise, we say that $f$ is not extendable at $z_0$.
\end{defi}

Below we will use the above definition of extendability.

\begin{prop}
Let $L$ be a compact subset of $ \mathbb{C}$ and $U$ be an open subset of $ \mathbb{C}$, such that $L\subseteq U$. Let also $ M $ and $r$ be positive real numbers, $p\in \{0,1,2, \cdots\}\cup \{\infty\}$ and $ z_0 \in \partial L $. The set $$E_{M,p,U,L,z_0,r}=\{f\in A^p(U\setminus L): there\ exists\ F\in H(D(z_0, r))\ such$$ $$ that\ F|_{(U\setminus L)\cap D(z_0,r)}=f|_{(U\setminus L)\cap D(z_0,r)} \ and\ \\
\parallel F\parallel_{\infty}=\sup\limits_{z\in D(z_0,r)}|F(z)|\leq M \}$$ is a closed subset of $A^p(U\setminus L)$. Also, if there exists $f_0\in A^p(U\setminus L)$ which is not extendable at $z_0$, then the interior of $E_{M,p,U,L,z_0,r}$ is void in $A^p(U\setminus L)$.
\end{prop}

\begin{proof}
We will first prove that the set $E_{M,p,U,L,z_0,r}$ is a closed subset of $ A^p(U\setminus L)$. Let 
$ (f_n)_{n\geq 1} $ be a sequence in $E_{M,p,U,L,z_0,r}$ 
converging in the topology of $A^p(U\setminus L)$ to a function $ f\in A^p(U\setminus L)$. This implies that $f_n$ converges uniformly on $\overline{U\setminus L}$ to $f$ and that there exists a sequence $ (F_n)_{n\geq 1} $ in $H(D(z_0, r))$ such that $F_n|_{(U\setminus L)\cap D(z_0,r)}=f_n|_{(U\setminus L)\cap D(z_0,r)}$ and $\parallel F\parallel_{\infty}\leq M$ for every $ n\geq 1 $. 
By Montel's theorem there exists
a subsequence of $(F_n)$, $(F_{k_n})$, which converges uniformly on the compact subsets of
$D(z_0,r)$ to a function $F$ which is holomorphic and bounded by $M$ on $D(z_0,r)$. 
Since $F_{k_n}$ converges 
to $f$ on $(U\setminus L)\cap D(z_0,r)$, the functions $f$ and $F$ are equal on $(U\setminus L)\cap D(z_0,r)$. Thus $f$ belongs to $E_{M,p,U,L,z_0,r}$ and $E_{M,p,U,L,z_0,r}$ is a closed subset of $ A^p(U\setminus L)$.\\
If there exists $f_0\in A^p(U\setminus L)$ which is not extendable at $z_0$,  the interior of $E_{M,p,U,L,z_0,r}$ is void in $A^p(U\setminus L)$, the proof of which is similar to the proof of Theorem $7.1$.
\end{proof}

Here we have one more dichotomy which is a local version of the first one.\\

\begin{theorem} Let $L$ be a compact subset of $ \mathbb{C}$ and $U$ be an open subset of $ \mathbb{C}$, such that $L\subseteq U$ and let $p\in \{0,1,2, \cdots\}\cup \{\infty\}$, $ z_0 \in \partial L $. The set  
$E_{p,U,L,z_0}=\bigcup\limits_{M=1}^{\infty}\bigcup\limits_{n=1}^{\infty} E_{M,p,U,L,z_0,\frac{1}{n}}$ is the set of extendable functions of $A^p(U\setminus L)$ at $ z_0 $. Then 
\begin{enumerate}[(i)]
\item either every function $f\in A^p(U\setminus L)$ is extendable at $ z_0 $
\item or generically all functions $f\in A^p(U\setminus L)$ are not extendable at $ z_0 $.
\end{enumerate}
\end{theorem}

\begin{proof}
If $(i)$ is not true, then Proposition $7.7$ shows that $E_{M,p,U,L,z_0,\frac{1}{n}}$ is closed with empty interior for all natural numbers $ n\geq 1, M\geq 1  $. Then
$$A^p(U\setminus L)\setminus E_{p,U,L,z_0}=\bigcap\limits_{M=1}^{\infty}\bigcap\limits_{n=1}^{\infty} (A^p(U\setminus L)\setminus E_{M,p,U,L,z_0,\frac{1}{n}})$$
is the intersection of a countable number of open and dense subsets of $A^p(U\setminus L)$ and Baire's Theorem shows that $A^p(U\setminus L)\setminus E_{p,U,L,z_0}$, which coincides with the set of non extendable functions of $A^p(U\setminus L)$ at $ z_0 $, is a dense and $G_{\delta}$ subset of $A^p(U\setminus L)$.
\end{proof}

In what follows we compare two notions: local extendability and existence of a holomorphic extension. At first we examine the case of a compact set $L$ with empty interior.

\begin{prop}
Let $L$ be a compact subset of $ \mathbb{C}$ and $U$ be an open subset of $ \mathbb{C}$, such that $L\subseteq U$ and $L^{\circ}=\emptyset$. Let also  $f\in H(U\setminus L)$. Then $f$ is extendable at every $z_0\in \partial L $ if and only if there exists a holomorphic extension $F$ of $f$ on $U$. If additionally $ f\in  A^p(U\setminus L)$ for some $p\in\{0,1,\dots\}\cup\{\infty\}$, then $ F\in A^p(U) $.
\end{prop}

\begin{proof}
If there exists a holomorphic extension $F$ of $f$ on $U$, then obviously $f$ is extendable at every $z_0\in \partial L=L $.

Conversely, if $f$ is extendable at every $z_0\in L $, then for every $z_0\in L $ there exist a positive real number $r_{z_0}$ and a holomorphic function $F_{z_0}$ on $ D(z_0,r_{z_0}) $ such that $D(z_0,r_{z_0})\subseteq U$ and $ F_{z_0}|_{(U\setminus L)\cap D(z_0,r_{z_0})}=f|_{(U\setminus L)\cap D(z_0,r_{z_0})} $. Let $z_1,z_2 \in L$ such that $V=D(z_1,r_{z_1})\cap D(z_2,r_{z_2})\neq \emptyset$. Since $L^{\circ}=\emptyset$, $ V\setminus L $ is a non-empty, open set. Thus $F_{z_1}, F_{z_2}$ are holomorphic on the domain $V$ and coincide with $f$ on $ V\setminus L $. By analytic continuation,  $F_{z_1}=F_{z_2}$ on $V$. So, the function $F$ defined on $ U $ such that $F(z)=F_{z}(z) $ for every $z\in L$ and $F(z)=f(z) $ for every $z\in U\setminus L$ is a holomorphic extension of $ f $ on $U$. Obviously, if $ f\in  A^p(U\setminus L)$, then $ F\in A^p(U) $.
\end{proof}

\begin{remark}
If $L^{\circ}\neq\emptyset$ the equivalence at Proposition $7.9$ is not true. Indeed if $w\in L^{\circ}\neq\emptyset$, then the holomorphic function $ f(z)=\frac{1}{z-w} $ for $z\in U\setminus L$ can not be extended to a holomorphic function on $U$, but it is extendable at every $ z_0\in \partial L $.
\end{remark}

We consider again a compact set $L\subseteq \mathbb{C}$ and an open set $U\subseteq \mathbb{C}$, such that $L\subseteq U$ and a $p\in \{0,1,2, \cdots\}\cup \{\infty\}$. 
Now we want to find a similar connection between $a_p(L)$ and $a_p(L\cap \overline{D(z_0, r)})$; that is, is the condition $a_p(L)=0$ equivalent to the condition $a_p(L\cap \overline{D(z_0, r)})=0$ for all $ z_0\in  L $?

If we suppose that $ L^{\circ}\neq\emptyset $, then there exist $ z_0$ and $r>0$ such that $D(z_0, r)\subseteq L$. Thus $ a_p(L)$ and $a_p(L\cap \overline{D(z_0, r)})$ are strictly positive.

So, we do not need to assume that $ L^{\circ}=\emptyset $, since it follows from both the conditions $a_p(L)=0$ and $a_p(L\cap \overline{D(z_0, r)})=0$ for every $ z_0\in  L $ and for some $r=r_{z_0}>0$. 
Also, the first condition obviously implies the second one.

Probably Theorem $3.6$ holds even for $p\geq 1$. Specifically, if $a_p(L)=0$ and $V$ is an open set, then every function $g\in A^p(V\setminus L)$ belongs to $A^p(V)$.
This leads us to believe that the above conditions are in fact equivalent. However, this will be examined in future papers.

\medskip

\noindent \textbf{Acknowledgement}: We would like to thank A. Borichev, P. Gauthier, J.-P. Kahane, V. Mastrantonis, P. Papasoglu and A. Siskakis for helpful 
communinications.

\noindent
Eleftherios Bolkas, Vassili Nestoridis, Christoforos Panagiotis\\

\noindent
University of Athens\\
Department of Mathematics\\
157 84 Panepistemiopolis\\
Athens\\
Greece\\

\noindent
E-mail addresses:\\
lefterisbolkas@gmail.com (Eleftherios Bolkas)\\
vnestor@math.uoa.gr (Vassili Nestoridis)\\
chris$\_$panagiwtis@hotmail.gr (Christoforos Panagiotis)
\\

\noindent
Michael Papadimitrakis\\

\noindent
University of Crete\\
Department of Mathematics and Applied Mathematics\\
Voutes Campus, GR-700 13, Heraklion\\
Crete\\
Greece\\

\noindent
E-mail address:\\
papadim@math.uoc.gr (Michael Papadimitrakis)\\

\noindent
Current addresses:\\
Bolkas Eleftherios\\
Princeton University, Department of Mathematics, Fine Hall, Washington Road, Princeton, NJ 08544-1000, USA\\
ebolkas@math.princeton.edu\\

\noindent
Panagiotis Christoforos\\
Mathematics Institute, University of Warwick, CV4 7AL, UK\\
C.Panagiotis@warwick.ac.uk\\

\end{document}